\RequirePackage{fix-cm}
\documentclass[smallextended,table,x11names,
    envcountsect,
    envcountsame,
    envcountreset]{svjour3}       
\smartqed  
\usepackage{graphicx}
\usepackage[T1]{fontenc}
\usepackage[utf8x]{inputenc}
\usepackage[english]{babel}
\usepackage{xcolor}
\usepackage{amsmath, amssymb}
\usepackage{dsfont}
\usepackage{mathtools}
\usepackage{booktabs}
\usepackage{tabularx}
\usepackage[noend, ruled, noline, linesnumbered]{algorithm2e}
\usepackage{array}
\usepackage{enumerate}
\newcolumntype{L}[1]{>{\raggedright\let\newline\\\arraybackslash\hspace{0pt}}m{#1}}
\newcolumntype{C}[1]{>{\centering\let\newline\\\arraybackslash\hspace{0pt}}m{#1}}
\newcolumntype{R}[1]{>{\raggedleft\let\newline\\\arraybackslash\hspace{0pt}}m{#1}}

\newcommand{\norm}[1]{\Vert#1\Vert}
\newcommand{\Norm}[1]{\left\Vert#1\right\Vert}
\newcommand{\abs}[1]{\vert#1\vert}
\newcommand{\Abs}[1]{\left\vert#1\right\vert}
\newcommand{\set}[1]{\left\{#1\right\}}

\newcommand{\R}{\mathds{R}}
\newcommand{\N}{\mathds{N}}
\newcommand{\One}{\mathds{1}}
\newcommand{\st}{\mathrm{s.t.}}
\newcommand{\define}{\coloneqq}

\DeclareMathOperator{\diag}{Diag}
\DeclareMathOperator{\sign}{sign}
\DeclareMathOperator{\Sign}{Sign}
\DeclareMathOperator*{\argmin}{argmin}

\begin{document}

\title{A Primal-Dual Homotopy Algorithm for $\ell_{1}$-Minimization with $\ell_{\infty}$-Constraints\thanks{
    This material was based upon work partially supported by the National Science Foundation under Grant DMS-1127914 to the Statistical and Applied Mathematical Sciences Institute. 
    Any opinions, findings, and conclusions or recommendations expressed in this material are those of the author(s) and do not necessarily reflect the views of the National Science Foundation.}
}


\author{Christoph Brauer\and
  Dirk A. Lorenz\and
  Andreas M. Tillmann
}


\institute{Christoph Brauer \at
              Technische Universität Braunschweig, Institut für Analysis und Algebra\\
              Pockelsstr. 14, 38106 Braunschweig, Germany\\
              Tel.: +49-531-3917421\\
              Fax: +49-531-3917414\\
              \email{ch.brauer@tu-braunschweig.de}\\
              \and
              Dirk A. Lorenz \at
              Technische Universität Braunschweig, Institut für Analysis und Algebra\\
              Pockelsstr. 14, 38106 Braunschweig, Germany\\
              \email{d.lorenz@tu-braunschweig.de}\\
              \and
              Andreas M. Tillmann \at
              Technische Universität Darmstadt, AG Optimierung\\
              Dolivostr. 15, 64293 Darmstadt, Germany\\
              \email{tillmann@mathematik.tu-darmstadt.de}
}

\date{Received: date / Accepted: date}

\maketitle

\begin{abstract}
  In this paper we propose a primal-dual homotopy method for
  $\ell_1$-minimization problems with infinity norm constraints in the
  context of sparse reconstruction. The natural homotopy parameter is
  the value of the bound for the constraints and we show that there
  exists a piecewise linear solution path with finitely many break
  points for the primal problem and a respective piecewise constant
  path for the dual problem. We show that by solving a small linear
  program, one can jump to the next primal break point and then,
  solving another small linear program, a new optimal dual solution is
  calculated which enables the next such jump in the subsequent
  iteration. Using a theorem of the alternative, we show that the
  method never gets stuck and indeed calculates the whole path in a
  finite number of steps.
  
  Numerical experiments demonstrate the effectiveness of our algorithm. 
  In many cases, our method significantly outperforms commercial LP
  solvers; this is possible since our approach employs a sequence of
  considerably simpler auxiliary linear programs that can be solved
  efficiently with specialized active-set strategies. 

\end{abstract}

\section{Introduction}\label{sec:introduction}

With the advent of Compressed Sensing
\cite{CandesTao2005,Donoho2006,EldarKutyniok2012,FoucartRauhut2013},
recovery of sparse vectors by means of the popular Basis Pursuit
approach~\cite{ChenDonohoSaunders1998},
\begin{equation}
  \label{eq:bp}
  \tag{BP}
  \min_{x\in\R^n} \norm{x}_1\quad\st\quad Ax=b,
\end{equation}
and the so-called Basis Pursuit Denoising (or $\ell_1$-regularized
Least-Squares) problem
\begin{equation}
  \label{eq:l1ls}
  \tag{$\ell_1$-LS}
  \min_{x\in\R^n} \lambda\norm{x}_1 + \tfrac12\norm{Ax-b}_2^2,
\end{equation}
with $A\in\R^{m\times n}$, $b\in\R^m$ and $\lambda>0$, received a lot
of attention both theoretically and algorithmically over the past
decade (see, e.g., \cite{LorenzPfetschTillmann2015,FoucartRauhut2013}
and many references therein). However, the related problem
\begin{equation}
  \label{eq:p_delta}
  \tag{P$_{\delta}$}
  \min_{x\in\R^n} \norm{x}_1\quad\st\quad \norm{Ax-b}_\infty\leq\delta
\end{equation}
appears to be much less investigated. This problem can be rewritten as
a linear program (LP) by formulating the $\ell_\infty$-norm constraint
as linear inequalities and performing the usual variable split of $x$
into its positive and negative parts (see~\eqref{eq:p_deltaLP}
below). Thus, in principle, every LP solver can be applied to solve the
problem. However, in practice it may happen that the problem instances
are very large (and with $A$ dense or perhaps only available
implicitly) so that current LP solvers may not be able to handle the
problem well. Moreover, there are cases in which one does not only
want to solve the problem for a given instance of $(A,b,\delta)$ but
for a whole range of parameters~$\delta$.

Our interest in sparse approximation under $\ell_\infty$-constraints
via the problem~\eqref{eq:p_delta} is motivated by several practical
applications:
\begin{itemize}
\item The \emph{Dantzig selector} problem~\cite{CandesTao2007}
  \begin{equation}
    \label{eq:ds}
    \tag{DS$_{\delta}$}
    \min_{x\in\R^n} \norm{x}_1\quad\st\quad \Norm{A^\top(Ax-b)}_\infty\leq\delta
  \end{equation}
  is a special case of~\eqref{eq:p_delta} and has numerous applications
  in statistical estimation, see, e.g.,~\cite{ZhengLiu2011}, where the
  whole solution path for~$\delta>0$ is computed as a selection step
  prior to a classification step.
\item In \emph{sparse dequantization}, one has quantized measurements
  $b=Q(A\bar x)$ of some signal vector $\bar x$ which is assumed to be
  sparse. If the quantization level is known, one can
  interpret~\eqref{eq:p_delta} as the problem of finding a
  reconstruction~$x^*$ with minimal $\ell_1$-norm for which the
  measurements~$Ax^*$ produce the same quantized measurements~$b$. We
  refer to~\cite{JaquesHammondFadili2011} for the general idea and
  to~\cite{BrauerGerkmannLorenz2016} for a recent application to
  speech processing.
\item In \emph{sparse linear discriminant analysis} as proposed
  in~\cite{CaiLiu2011}, one obtains a problem of the
  form~\eqref{eq:p_delta} in which $A$ is a sample covariance matrix
  and $b$ is a difference of sampled means. Similarly, the so-called
  CLIME estimator~\cite{CaiLiuLuo2011} solves \emph{sparse precision
    matrix estimation} problems via a sequence of \eqref{eq:p_delta}
  problems in each of which $A$ is again a covariance matrix and $b$
  is equal to a unit vector.
\end{itemize}

In this paper, we develop a homotopy algorithm for the
problem~\eqref{eq:p_delta}. The starting point is that for $\delta\geq
\norm{b}_\infty$, the vector $x=0$ is obviously the optimal
solution. Moreover, we will show that for a solution~$x$
of~\eqref{eq:p_delta} for a given~$\delta>0$, there exist a
direction~$d$ and a scalar $t_0>0$ such that $x+td$ is a solution of
(P$_{\delta - t}$) for $0\leq t\leq t_0$. Our algorithm builds on
these observations and calculates a path of solutions for decreasing
values of~$\delta$ until a target $\delta$-value is reached; we shall
prove that the algorithm is able to compute such a path in finitely many steps
(even if the final value is $\delta=0$). Our approach resembles the
popular homotopy method for \eqref{eq:l1ls},
cf.~\cite{OsbornePresnellTurlach2000}, but, as detailed later, our
method has to work on both the primal and dual problem simultaneously,
so that the algorithms differ considerably.

The remainder of this paper is structered as follows: We further touch
upon related methods in Subsection~\ref{subsec:relatedwork} below, and
fix some notation in Subsection~\ref{subsec:notation}. The main part
of the paper, Section~\ref{sec:homotopyalgorithm}, constitutes a
detailed derivation of our homotopy approach to~\eqref{eq:p_delta},
including theoretical results on iterative improvement and finite
termination. An efficient solution approach for subproblems
encountered in our scheme is put forth in
Section~\ref{sec:practicalconsiderations}. We consider some practical
applications and present computational results in
Section~\ref{sec:applicationsandexamples}, discuss possible
extensions and conclude the paper in Section~\ref{sec:extensions}.

\subsection{Related Work}\label{subsec:relatedwork}
Homotopy concepts have been around for decades, so it should come as
no surprise that our approach bears some resemblance to several
earlier algorithms. In the following, we briefly comment on
similarities and differences with respect to the arguably most
naturally related methods.

\subsubsection{Parametric Simplex Method}\label{subsec:PSM}
It is well-known that problem~\eqref{eq:p_delta} can be recast as an LP,
e.g., 
\begin{align}
  \nonumber \min_{x^\pm,s^\pm\in\R^n}\quad &\One^\top x^+ + \One^\top x^-\\
  \label{eq:p_deltaLP}  \text{s.t.}\quad & \left(\begin{array}{rrrr}A & -A & ~I & 0\\-A & A & 0 & ~I\end{array}\right)\left(\begin{array}{c}x^+\\x^-\\s^+\\s^-\end{array}\right)=\left(\begin{array}{r}b+\delta\One\\-b+\delta\One\end{array}\right)\\
  \nonumber  & x^+,~x^-,~s^+,~s^-\geq 0.
\end{align}
There exists a variety of homotopy schemes for LPs, see, for instance,
\cite{Blum1988,Nazareth1991} and references therein. In fact, the latter
work shows how many standard LP algorithms (simplex, affine-scaling
and interior-point methods) can be subsumed under a unifying homotopy
framework, exhibiting nice connections between intuitively very
different approaches. The LP homotopy method most naturally related to
our approach results from treating the parameter $\delta$ itself as
the homotopy parameter (as we shall also do in our method) in the
above LP---the so-called (self-dual) \emph{parametric simplex method}
(PSM) \cite{Dantzig1963,Vanderbei2001}. Very briefly, PSM perturbs both
the LP right-hand side and objective coefficient vectors using the
same parameter and then drives this parameter down to zero, performing
primal or dual simplex pivot steps at each breakpoint in the
(piecewise linear) parameter homotopy path. For a sufficiently large
initial parameter, a primal-dual feasible (hence, optimal) basis is
easily found and used to start the algorithm; reducing the parameter,
basis optimality is maintained until either a basic variable or
nonbasic reduced cost coefficient changes sign, which identifies the
breakpoints and induces an appropriate simplex step to exchange some
basis element for a nonbasic one. (For a detailed formal description,
see, e.g., \cite[pp.~115--121]{Vanderbei2001}.)

In fact, PSM was very recently proposed for sparse linear discriminant
analysis problems by means of reformulating the associated
problem~\eqref{eq:p_delta} as precisely the LP stated above,
see~\cite{PangZhaoVanderbeiLiu2015}, in which PSM is applied to several
other problems as well. For the above special parameterized LP, one
needs to stop PSM as soon as the parameter drops below the target
original $\delta$ (\emph{not} zero) and since the objective is
unperturbed, only primal simplex pivot steps are performed throughout
the entire algorithmic process (i.e., each breakpoint identifies some
variable that is to leave the basis in exchange for a nonbasic one;
neither of these facts is mentioned in~\cite{PangZhaoVanderbeiLiu2015}).

If the optimal solutions for each respective parameter interval are
unique, then PSM and our approach necessarily produce the same
solution path. However, the paths may differ if multiple optimal
solutions occur, as the underlying algorithmic concepts are different:
For one thing, we operate in the original variable space ($n$ primal
and $m$ dual variables versus $2n+2m$ variables in the above
parameterized LP), and thus avoid doubling the dimensions. Moreover,
in each iteration, PSM is restricted to moving to an adjacent basis
and, in particular, can get ``stuck'' at a certain parameter value for
several iterations (namely when several pivot steps are needed to
eventually arrive at a new basis that allows to further reduce the
parameter). Such a situation can never occur in our algorithm
(cf.~Theorem~\ref{theorem:alternatives} in
Section~\ref{subsec:theoremsofthealternative} below); indeed, our scheme
guarantees the largest reduction of $\delta$ in every iteration and
moves directly to associated optimal points.

Regarding implementation, PSM is subject to all advantages and
drawbacks that come with any simplex method, e.g., its basic version
(as described in~\cite{Vanderbei2001}) may cycle and hence not even
terminate, special care needs to be taken to compute and maintain
numerically stable basis matrix factorizations, etc. Our approach is
straightforward to implement, but requires access to an LP solver for
subproblem optimization---given the large selection of sophisticated
LP solvers (both proprietary and freely available) to choose from, we
actually consider this a feature, not a disadvantage. In particular,
this allows us to use a certain active-set LP strategy that turns out
to be particularly well-suited to the subproblems occurring during our
method, see Section~\ref{sec:practicalconsiderations}. At least in
case of multiple optimal solutions, both PSM and our homotopy method
are naturally influenced by choices made for crucial steps
(i.e., pivoting rules for PSM and LP subproblem solver choice in our
implementation), which makes a direct numerical comparison somewhat
meaningless; hence, we do not delve into this subject further. (It
should however be noted that the homotopy approach not only provides
the whole solution path, but for sparse solutions is also
significantly faster than applying a standard LP solver to the LP
reformulation of~\eqref{eq:p_delta} directly.)

Finally, let us remark that the relationship between
\eqref{eq:p_delta} and linear programming extends, in a sense, both
ways: Obviously, a general LP method can be used to
solve~\eqref{eq:p_delta}, rewritten as the above LP, but a relevant
and relatively large subclass of LPs can also be recast into a form
resembling~\eqref{eq:p_delta} for which our algorithm can be adapted
straightforwardly, cf. Section~\ref{sec:extensions}.

\subsubsection{Dantzig Selector and $\ell_1$-Regularized Least-Squares Homotopy}\label{subsec:DSl1LS}
A homotopy scheme for the Dantzig selector problem~\eqref{eq:ds} was
proposed in~\cite{AsifRomberg2009}. There, the general idea is also to
perform primal and dual update steps in each iteration, starting from
a large value for the parameter $\delta$ (for which the optimal
solution is trivially known) and driving it down toward the desired
level. The update steps consist of finding directions along which
optimality conditions are maintained and by choosing suitable step
sizes, breakpoints in the homotopy path are identified; the supports
of the current primal and dual variables are updated one element at a
time\footnote{The description in~\cite{AsifRomberg2009} is a bit
  unclear in this regard; it seems the authors implicitly use a kind
  of subproblem uniqueness assumption under which this works out well,
  although the choice of indices entering or leaving a support
  apparently needs not be uniquely determined in general. Also, they
  claim the optimality conditions they work with imply uniqueness, but
  they are equivalent to the standard LP optimality conditions with
  strict complementary slackness (see, e.g.,
  \cite[Section~7.9]{Schrijver1986}) applied to the LP obtainable
  from~\eqref{eq:ds}, which do not import a statement about
  uniqueness.}.

Clearly, \eqref{eq:ds} is a special case of the more general
problem~\eqref{eq:p_delta} we consider. Moreover, we allow primal and
dual supports to change by more than one component per iteration (and
we do not make any uniqueness assumptions), so our approach also
generalizes that of~\cite{AsifRomberg2009} conceptually. Another
difference is that we do not explicitly compute directions first but
directly obtain the respective next points. Nevertheless, the method
from~\cite{AsifRomberg2009} remains of interest in its own right,
since the special (Gramian) structure of the constraint matrix allows
for a more direct subproblem treatment than the LPs we will solve.

As discussed in \cite{JamesRadchenkoLv2009,AsifRomberg2010}, for
certain sparsity levels of the optimal solution to~\eqref{eq:ds}
and/or conditions on the matrix $A$, the whole respective solution
paths of the Dantzig selector homotopy from~\cite{AsifRomberg2009},
the related but different DASSO algorithm
from~\cite{JamesRadchenkoLv2009}, and the homotopy scheme
for~\eqref{eq:l1ls} (see~\cite{OsbornePresnellTurlach2000})
coincide. (Also, the Dantzig selector homotopy algorithm can be
modified quite simply to reduce to the $\ell_1$-LS homotopy scheme,
cf.~\cite{AsifRomberg2009}).

Thus, our algorithm is naturally related to those methods as well:
Though \eqref{eq:p_delta} generalizes \eqref{eq:ds}, which in turn is
sometimes equivalent to \eqref{eq:l1ls}, neither problems are
equivalent, whence the various algorithms are necessarily different,
though certainly very similar in spirit. It is also worth noting that
while the homotopy for~\eqref{eq:l1ls} is a primal
method\footnote{More precisely, due to the smooth $\ell_2$-part
  in~\eqref{eq:l1ls}, for every primal optimal solution w.r.t. some
  parameter $\delta$, the associated dual optimal solution is known in
  closed-form, which can be substituted into the algorithmic formluae
  directly, eliminating the need for keeping a dual variable
  explicitly.}, the approaches for~\eqref{eq:ds} and also our proposed
algorithm work in a primal-dual fashion.

\subsection{Notation}\label{subsec:notation}
For $A\in\R^{m\times n}$, $a_{i}^{\top}$ denotes the $i$-th row and
$A_{j}$ denotes the $j$-th column of $A$. Moreover, for
$I\subseteq\set{1,\dots,m}$ and $J\subseteq\set{1,\dots,n}$,
$A_{J}^{I}$ denotes the sub-matrix of $A$ with rows indicated by $I$
and columns indicated by $J$. Sometimes, we write $A_{J}^{\top} =
(A_{J})^{\top}$.

By $\odot$, we denote the component-wise product of two vectors, i.e., for $x, z\in\R^{n}$, we have $(x\odot z)_{j} = x_{j}z_{j}$.

Furthermore, we define $\diag(x)$ to be the
$n\times n$ diagonal matrix having the entries of the vector~$x$ as
its diagonal elements.

As usual, $\norm{\cdot}_{1}$ and $\norm{\cdot}_{\infty}$ denote
the respective norms, i.e.,
\begin{equation*}
  \norm{x}_{1} = \sum_{j = 1}^{n}\abs{x_{j}} \quad \text{and} \quad \norm{x}_{\infty} = \max_{j = 1,\dots,n}\abs{x_{j}}.
\end{equation*}
The subdifferential of $\norm{\cdot}_{1}$ at $x$ is denoted by
\begin{equation*}
  \Sign(x) \define \partial\norm{\cdot}_{1}(x) = \big\{\xi\in [-1, 1]^{n} : x_{j} \neq 0 \Rightarrow \xi_{j} = \sign(x_{j})\big\}. 
\end{equation*}

Finally, for given primal variable $x$, dual variable $y$ and bound
$\delta$, we introduce the index sets
\begin{align*}
  J_{P} \define &\set{j : x_{j} \neq 0} &&\text{(primal support)},\\
  I_{P} \define &\set{i : \Abs{a_{i}^{\top}x - b_{i}} = \delta} &&\text{(primal active set of constraints)},\\
  J_{D} \define &\set{j : \Abs{A_{j}^{\top}y} = 1} &&\text{(dual active set)}\\
 \text{and}\quad I_{D} \define &\set{i : y_{i} \neq 0} &&\text{(dual support)},
\end{align*}
cf.~\eqref{eq:p_delta} and its dual problem~\eqref{eq:d_delta}
(defined below). Generally, for notational simplicity, we do not make
the sets' dependency on $x$, $y$ and $\delta$ explicit as it will be
clear from the context. Nevertheless, if we consider these index sets
for specific algorithmic iterates $x^{k}$, $y^{k}$ and $\delta^{k}$,
we write $J_{P}^{k}$, $I_{P}^{k}$, $J_{D}^{k}$, and $I_{D}^{k}$,
respectively. 

Set complements are denoted by a superscript $c$ and always pertain to
the respective natural superset; e.g., $J_P^c=\set{1,\dots,n}\setminus
J_P$ and $I_P^c=\set{1,\dots,m}\setminus I_P$.

\section{Homotopy Algorithm}\label{sec:homotopyalgorithm}
In the following, we describe our algorithmic approach in detail and
prove its correctness and finite convergence. A pseudocode of the
method is given in Algorithm~\ref{alg:homotopy_iteration} below.  With
a wink and a nod to a certain well-known basis pursuit solver, we call
our algorithm $\ell_1$-\textsc{Houdini} ($\ell_1$-norm HOmotopy UnDer
Infinity-Norm constraInts). Throughout, we assume w.l.o.g. that
$\delta<\norm{b}_\infty$ (otherwise, $x^*=0$ trivially
solves~\eqref{eq:p_delta}).

\subsection{Optimality Conditions and Algorithmic Idea}\label{sec:optimalitysystemandalgorithmicscheme}

It is well-known that $x^{*}$ is an optimal solution of
\eqref{eq:p_delta} if and only if there exists a $y^{*}$ such that
\begin{align}
  -A^\top y^{*} &\in \Sign(x^{*})\label{eq:oc_1}\\
 \text{and}\quad  Ax^{*} - b &\in \delta \Sign(y^{*}).\label{eq:oc_2}
\end{align}
In particular, such a $y^*$ is an optimal solution to the dual problem
of~\eqref{eq:p_delta},~i.e.,
\begin{equation}
  \label{eq:d_delta}
  \tag{D$_{\delta}$}
  \max_{y\in\R^{m}} \ -b^{\top}y - \delta\norm{y}_{1} \quad \st \ \Norm{A^{\top}y}_{\infty} \leq 1.
\end{equation}
Thus, we call $y^{*}$ a \emph{dual certificate} and $(x^{*}, y^{*})$
an \emph{optimal pair} for \eqref{eq:p_delta}.  In particular, any
optimal pair satisfies $\norm{x^{*}}_{1} =
-b^{\top}y^{*}-\delta\norm{y^{*}}_{1}$, i.e., the primal and the dual
problem attain the same optimal value. Note that, as a consequence of
the optimality conditions \eqref{eq:oc_1} and \eqref{eq:oc_2}, it
always holds that $J_{P}\subseteq J_{D}$ and $I_{D}\subseteq I_{P}$ in
case $(x^{*}, y^{*})$ is an optimal pair. 

Our approach is to find an optimal pair by repeatedly making use
of~\eqref{eq:oc_1} and~\eqref{eq:oc_2}. Instead of solving
(\ref{eq:p_delta}) directly, we start by setting $\delta^{0} \define
\norm{b}_{\infty}$ and observe that $x^{0} = 0$ is an optimal solution
of (P$_{\delta^{0}}$). Now, the main idea behind the iterations of our
method is the following: Let $k\in\N_{0}\define\N\cup\{0\}$ and $(x^{k}, y^{k})$ be an
optimal pair for (P$_{\delta^{k}}$). First, we seek a $y^{k+1}\neq
y^{k}$ such that $(x^{k}, y^{k+1})$ is still an optimal pair for
(P$_{\delta^{k}}$). After that, we aim at identifying $x^{k+1}$ and
$t>0$ such that with $\delta^{k+1} = \delta^{k} - t$, $(x^{k+1},
y^{k+1})$ is an optimal pair for (P$_{\delta^{k+1}}$). We repeat these
steps as long as $\delta^{k+1} > \delta$; when finally $\delta^{k+1} =
\delta$, we have found an optimal pair $(x^{k+1}, y^{k+1})$ for our
initial problem \eqref{eq:p_delta}.

We remark that while \eqref{eq:oc_1} and \eqref{eq:oc_2} show that,
e.g., $y^{0} = 0$ would be a valid dual certificate associated with
$x^0$ (other similarly simple choices are possible), such a heuristic
choice---then to be used for a first primal update step---may lead to
a ``zero step'' ($t=0$, $x^{1}=x^0$), after which a new dual iterate
must be computed. Therefore, in $\ell_1$-\textsc{Houdini}, we will
actually start with the computation of a dual certificate directly
(i.e., we do not need any $y^0$).

\subsection{Primal Updates}\label{subsec:primalupdates}
Suppose $(x^{k}, y^{k+1})$ is an optimal pair for
(P$_{\delta^{k}}$) and we seek $x^{k+1}$ and $t$ such that
$(x^{k+1}, y^{k+1})$ is an optimal pair for (P$_{\delta^{k}-t}$). From
\eqref{eq:oc_1} and \eqref{eq:oc_2} we know that $x^{k+1}$ and $t$
must fulfill
\begin{equation*}
  -A^{\top}y^{k+1} \in \Sign(x^{k+1}) \quad \text{and} \quad Ax^{k+1} - b \in (\delta^{k}-t) \Sign(y^{k+1}).
\end{equation*}
The first condition restricts both the support and the sign of~$x^{k+1}$, i.e., it must hold that
\begin{align*}
  x^{k+1}_{j} &= 0 \qquad\text{if}\quad\abs{(A^{\top}y^{k+1})_{j}} < 1,\\
  x^{k+1}_{j} &\geq 0 \qquad\text{if}\quad(A^{\top}y^{k+1})_{j} = -1\\
  \text{and}\quad x^{k+1}_{j} &\leq 0 \qquad\text{if}\quad(A^{\top}y^{k+1})_{j} = 1,
\end{align*}
or equivalently,
\begin{equation}
  \label{eq:primal_update_support_sign}
  x^{k+1}_{J_{D}^{c}} = 0 \quad \text{and} \quad (A_{J_{D}}^{\top}y^{k+1}) \odot x^{k+1}_{J_{D}} \leq 0.
\end{equation}
We split the second condition and start with the components $I_{D}$
in which $y^{k+1}$ is non-zero and thus, $\Sign(y_{I_{D}}^{k+1}) =
\sign(y_{I_{D}}^{k+1})$ is single-valued. This leads us to a linear
system in $x^{k+1}$ and $t$:
\begin{equation}
  \label{eq:primal_update_equations}
  A^{I_{D}}x^{k+1} + t\cdot\sign(y^{k+1}_{I_{D}}) = b + \delta^{k}\sign(y^{k+1}_{I_{D}}).
\end{equation}
The remainder of the second condition dictates the inclusions
\begin{equation*}
  a_{i}^{\top}x^{k+1} - b_{i} \in [-(\delta^{k}-t), \delta^{k} - t] \qquad\text{for all }i\in I_{D}^{c},
\end{equation*}
 which are equivalent to the linear constraints
\begin{equation}
  \label{eq:primal_update_inequations}
  -A^{I_{D}^{c}}x^{k+1} + t\One \leq \delta^{k}\One - b_{I_{D}^{c}} \qquad \text{and} \qquad A^{I_{D}^{c}}x^{k+1} + t\One \leq \delta^{k}\One + b_{I_{D}^{c}}.
\end{equation}
Finally, intuitive bounds for $t$ are given by
\begin{equation}
  \label{eq:primal_update_t}
  0 \leq t \leq \delta^{k} - \delta.
\end{equation}
Therein, the lower bound prevents regress and the upper bound ensures
that we do not jump over an optimal solution of the original problem
(recall that under our assumption $\delta<\norm{b}_\infty$, any
optimal solution of~\eqref{eq:p_delta} lies on the boundary of the
feasible set).

Note that, by construction, $x^{k+1} = x^{k}$ and $t = 0$ always yield
a solution of
\eqref{eq:primal_update_support_sign}--\eqref{eq:primal_update_t}. Nevertheless,
this choice would imply stagnation (the aforementioned ``zero
step''). In contrast, we can perform a maximal step with respect to
the current iterates $(x^{k}, y^{k+1})$ by maximizing~$t$ w.r.t. the
constraints
\eqref{eq:primal_update_support_sign}--\eqref{eq:primal_update_t},
which amounts to solving a linear program. (Note that the number of
variables is substantially reduced by eliminating
$x_{J_{D}^{c}}^{k+1}$, which must be zero; typically, $J_D$ will be
very small---and hence, $J_D^c$ large---at least in the beginning,
although generally this depends on the structure of~$b$.) 

In case the maximum objective is $t=0$, no progress is achievable by
performing a primal update; we will see later
(cf. Lemma~\ref{lemma:iterates}) that this case, in fact, never occurs
during our algorithm. Also, since $t=0$ is always possible, the lower
bound $t\geq 0$ is redundant and can be omitted
from~\eqref{eq:primal_update_t}.

\subsection{Dual Updates}\label{subsec:dualupdates}
The dual update follows the same principle as the primal update except
that here, $x^{k}$ and $\delta^{k}$ are fixed and we seek $y^{k+1}$
such that
\begin{equation*}
  -A^{\top}y^{k+1} \in \Sign(x^{k}) \quad \text{and} \quad Ax^{k} - b \in \delta^{k} \Sign(y^{k+1}).
\end{equation*}
Here, the second condition restricts the support and the sign of
$y^{k+1}$, i.e.,
\begin{equation}
  \label{eq:dual_update_support_sign}
  y^{k+1}_{I_{P}^{c}} = 0 \quad \text{and} \quad -\sign(A^{I_{P}}x^{k} - b_{I_{P}}) \odot y^{k+1}_{I_{P}} \leq 0.
\end{equation}
We split the first condition. Starting with the primal support
$J_{P}$, on which $\Sign(x^{k}_{J_{P}}) = \sign(x^{k}_{J_{P}})$ is
single-valued, we obtain the linear  system
\begin{equation}
  \label{eq:dual_update_equations}
  -A_{J_{P}}^{\top}y^{k+1} = \sign(x_{J_{P}}^{k}).
\end{equation}
On the complementary components $J_{P}^{c}$, the first condition
yields the linear constraints
\begin{equation}
  \label{eq:dual_update_inequations}
  -\One \leq A_{J_{P}^{c}}^{\top}y^{k+1} \leq \One.
\end{equation}

Just as in case of the primal update, there is a trivial solution to
\eqref{eq:dual_update_support_sign}--\eqref{eq:dual_update_inequations},
namely $y^{k+1} = y^{k}$. Moreover, we can again exploit that the
feasible support $I_{P}$ of $y^{k+1}$ will, at least in the beginning,
be small (so that many variables $y^{k+1}_{I_P^c}=0$). However, in
contrast to the primal update, where it was obvious to maximize~$t$,
it is not directly clear which solution we should prefer in case
\eqref{eq:dual_update_support_sign}--\eqref{eq:dual_update_inequations}
does not have a unique feasible point. The following theorem of
alternatives gives an answer to this problem.

\subsection{A Theorem of the Alternative}\label{subsec:theoremsofthealternative}
The following results provide, in particular, a selection rule for the
dual update which forms a key element for a working algorithm since it
guarantees the subsequent primal update to be successful (i.e., not a
``zero step''). 

The two alternatives \eqref{1.0}--\eqref{1.4} and
\eqref{2.1}--\eqref{2.4} in the lemma below are linear (in-)equality systems that
\emph{improvement directions} must obey (when interpreting primal and
dual updates as moving from $x^k$ to $x^{k}+t\cdot d$ and from $y^k$
to $y^k+s\cdot e$, respectively). 

\begin{lemma}
  \label{lemma:existence_of_directions}
  Let $(\hat x, \hat y)$ be an optimal pair for (P$_{\hat \delta}$)
  for some $0\leq \hat\delta<\norm{b}_\infty$. Then, one and only one
  of the systems
  \begingroup\allowdisplaybreaks
  \begin{subequations}
    \begin{alignat}{3}
      -\sign(A\hat x - b)^{\top}&e &\quad < &\quad 0\label{1.0}\\
      A_{J_{P}}^{\top}&e &\quad = &\quad 0\label{1.1}\\
      A_{J_{D}\setminus J_{P}}^{\top}\hat y \odot A_{J_{D}\setminus J_{P}}^{\top}&e &\quad \leq &\quad 0\label{1.2}\\
      -\sign(A^{I_{P}\setminus I_{D}}\hat x - b_{I_{P}\setminus I_{D}}) \, \odot \, &e_{I_{P}\setminus I_{D}} &\quad \leq &\quad 0\label{1.3}\\
      &e_{I_{P}^{c}} &\quad = &\quad 0\label{1.4}
    \end{alignat}
  \end{subequations}
  \endgroup
  and
  \begingroup\allowdisplaybreaks
  \begin{subequations}
    \begin{alignat}{3}
      A^{I_{D}}&d &\quad = &\quad -\sign(\hat y_{I_{D}})\label{2.1}\\
      \sign(A^{I_{P}\setminus I_{D}}\hat x - b_{I_{P}\setminus I_{D}}) \odot A^{I_{P}\setminus I_{D}} &d &\quad \leq &\quad -\One\label{2.2}\\
      A_{J_{D}\setminus J_{P}}^{\top}\hat y \, \odot \, &d_{J_{D}\setminus J_{P}} &\quad \leq &\quad 0\label{2.3}\\
      &d_{J_{D}^{c}} &\quad = &\quad 0.\label{2.4}
    \end{alignat}
  \end{subequations}
  \endgroup
  has a solution.
\end{lemma}
\begin{proof}
  With $\Sigma_{1} \define \diag(\sign(A^{I_{P}\setminus I_{D}}\hat x
  - b_{I_{P}\setminus I_{D}}))$ and $\Sigma_{2} \define
  \diag(A_{J_{D}\setminus J_{P}}^{\top}\hat y)$, we have $\Sigma_{1} =
  \Sigma_{1}^{-1}$ as well as $\Sigma_{2} = \Sigma_{2} ^{-1}$ and can
  rewrite the first system as
  \begin{alignat*}{5}
    -\One^{\top}\Sigma_{1}&e_{I_{P}\setminus I_{D}} &\quad -\sign(A^{I_{D}}\hat x - b_{I_{D}})^{\top}&e_{I_{D}} &\quad < &\quad 0\\
    -\Sigma_{2}  (A_{J_{D}\setminus J_{P}}^{I_{P}\setminus I_{D}})^{\top}&e_{I_{P}\setminus I_{D}} &\quad -\Sigma_{2}  (A_{J_{D}\setminus J_{P}}^{I_{D}})^{\top}&e_{I_{D}} &\quad \geq &\quad 0\\
    (A_{J_{P}}^{I_{P}\setminus I_{D}})^{\top}&e_{I_{P}\setminus I_{D}} &\quad +(A_{J_{P}}^{I_{D}})^{\top}&e_{I_{D}} &\quad= &\quad 0\\
    \Sigma_{1} &e_{I_{P}\setminus I_{D}} & & &\quad \geq &\quad 0.
  \end{alignat*}
  We substitute $\hat e_{I_{P}\setminus I_{D}} \define \Sigma_{1} e_{I_{P}\setminus I_{D}}$ and observe that the system has a solution if and only if the system
  \begin{alignat*}{5}
    -\One^{\top}&\hat e_{I_{P}\setminus I_{D}} &\quad - \sign(A^{I_{D}}\hat x - b_{I_{D}})^{\top}&e_{I_{D}} &\quad < &\quad 0\\
    -\Sigma_{2}  (A_{J_{D}\setminus J_{P}}^{I_{P}\setminus I_{D}})^{\top} \Sigma_{1} &\hat e_{I_{P}\setminus I_{D}} &\quad -\Sigma_{2}  (A_{J_{D}\setminus J_{P}}^{I_{D}})^{\top}&e_{I_{D}} &\quad \geq &\quad 0\\
    (A_{J_{P}}^{I_{P}\setminus I_{D}})^{\top}\Sigma_{1} &\hat e_{I_{P}\setminus I_{D}} &\quad +(A_{J_{P}}^{I_{D}})^{\top}&e_{I_{D}} &\quad = &\quad 0\\
    &\hat e_{I_{P}\setminus I_{D}} & & &\quad \geq &\quad 0
  \end{alignat*}
  is feasible. By Farkas' Lemma (see, e.g.,
  \cite[Corollary~7.1d]{Schrijver1986}), this system has a
  solution if and only if the associated alternative system
  \begin{alignat*}{5}
    -\Sigma_{1} A_{J_{D}\setminus J_{P}}^{I_{P}\setminus I_{D}} \Sigma_{2}  &\hat d_{J_{D}\setminus J_{P}}&\quad + \Sigma_{1} A_{J_{P}}^{I_{P}\setminus I_{D}}&d_{J_{P}} &\quad \leq &\quad -\One\\
    -A_{J_{D}\setminus J_{P}}^{I_{D}} \Sigma_{2} &\hat d_{J_{D}\setminus J_{P}} &\quad  + A_{J_{P}}^{I_{D}}&d_{J_{P}} &\quad = &\quad -\sign(A^{I_{D}}\hat x - b_{I_{D}})\\
    &\hat d_{J_{D}\setminus J_{P}} & & &\quad \geq &\quad 0
  \end{alignat*}
  is infeasible. Since $\sign(A^{I_{D}}\hat x -
  b_{I_{D}}) = \sign(\hat y_{I_{D}})$ and by substituting
  $d_{J_{D}\setminus J_{P}} \define -\Sigma_{2} \hat d_{J_{D}\setminus J_{P}}$, we obtain that equivalently, 
  \begin{alignat*}{3}
    \Sigma_{1} A_{J_{D}}^{I_{P}\setminus I_{D}} &d_{J_{D}} &\quad \leq &\quad -\One\\
    A_{J_{D}}^{I_{D}} &d_{J_{D}} &\quad = &\quad -\sign(\hat y_{I_{D}})\\
    -\Sigma_{2}  &d_{J_{D}\setminus J_{P}} &\quad \geq &\quad 0
  \end{alignat*}
  is infeasible.  The claim now follows by explicitly including
  $e_{I_P^c}=0$ and $d_{J_D^c}=0$ in the respective systems.  \qed
\end{proof}
In fact, our algorithm does not use explicit direction vectors, but
the above first set of alternative systems will be useful for the
proof of the next result and may also be of interest in its own
right. Below, note that in \eqref{3.0}--\eqref{3.5} and
\eqref{4.0}--\eqref{4.4}, we recognize the primal and dual update
conditions derived in the previous two subsections, respectively.

\begin{theorem}
  \label{theorem:alternatives}
  Let $(\hat x, \hat y)$ be an optimal pair for (P$_{\hat \delta}$)
  for some $0\leq\hat\delta<\norm{b}_\infty$. Then, the following four
  alternatives are equivalent.
  \begin{enumerate}[\normalfont (I)]
  \item\label{alt1} The system \eqref{1.0}--\eqref{1.4} is feasible.

  \item\label{alt2} The system \eqref{2.1}--\eqref{2.4} is infeasible.

  \item\label{alt3} $(\hat x, 0)$ is an optimal solution of
    \begin{subequations}
      \begin{alignat}{7}
        \max_{(x,t) \in \R^{n}\times\R} && & & &t\label{3.0}\\[-0.5em]
        \st~~~ && & &\quad A^{I_{D}}&x - b_{I_{D}} &\quad = &\quad (\hat \delta- t)\sign(\hat y_{I_{D}})\label{3.1}\\
        &&\quad(t-\hat \delta)\One &\quad \leq &\quad A^{I_{D}^{c}}&x - b_{I_{D}^{c}} &\quad \leq &\quad (\hat \delta- t)\One\label{3.2}\\
        && & &\quad A^{\top}\hat y\, \odot \, &x &\quad \leq &\quad 0\label{3.3}\\
        && & & &x_{J_{D}^{c}} &\quad = &\quad 0\label{3.4}\\
        && & & &t &\quad \leq &\quad \hat \delta - \delta\label{3.5}.
      \end{alignat}
    \end{subequations}
  \item\label{alt4} $\hat y$ is not an optimal solution of
    \begin{subequations}
      \begin{alignat}{6}
        \min_{y\in\R^{m}} && & &\quad-\sign(A\hat x - b)^{\top} &y\label{4.0}\\[-0.5em]
        \st  && & &\quad -A_{J_{P}}^{\top}&y &\quad = &\quad \sign(\hat x_{J_{P}})\label{4.1}\\
        &&\quad -\One &\quad \leq &\quad -A_{J_{P}^{c}}^{\top}&y &\quad \leq &\quad \One\label{4.2}\\
        && & &\quad -\sign(A\hat x - b) \, \odot \, &y &\quad \leq &\quad 0\label{4.3}\\
        && & & &y_{I_{P}^{c}} &\quad = &\quad 0\label{4.4}.
      \end{alignat}
    \end{subequations}
  \end{enumerate}
\end{theorem}
\begin{proof}
  Lemma \ref{lemma:existence_of_directions} already shows that
  alternatives~\eqref{alt1} and~\eqref{alt2} are equivalent.

  Moreover, since $(\hat x, \hat y)$ forms an optimal pair for
  (P$_{\hat \delta}$), several relations corresponding to constraints
  in the optimization problems of alternatives~\eqref{alt3}
  and~\eqref{alt4} already hold true, by the optimality conditions and
  the definitions of the index respective sets: Due
  to~\eqref{eq:oc_1}, \eqref{4.1} and \eqref{4.2} are satisfied, and
  due to~\eqref{eq:oc_2}, so are \eqref{3.1} and \eqref{3.2} for
  $t=0$, i.e., we have
  \begin{align*}
    -A_{J_{P}}^{\top}\hat y &&= &&\sign(\hat x_{J_{P}}), &&-\One &\leq &&-A_{J_{P}^{c}}^{\top}\hat y &&\leq &&\One,\\
    A^{I_{D}}\hat x - b_{I_{D}} &&= &&\hat \delta\sign(\hat y_{I_{D}}), &&-\hat\delta\One &\leq &&A^{I_{D}^{c}}\hat x - b_{I_{D}^{c}} &&\leq &&\hat\delta\One.
  \end{align*}
  By definition of the active sets $I_P$ and $J_D$ together with
  \eqref{eq:oc_1} and \eqref{eq:oc_2} (in other words, by
  complementary slackness) , \eqref{3.4} and \eqref{4.4} are also
  satisfied, i.e., $\hat x_{J_{D}^{c}}=0$ and $\hat
  y_{I_{P}^{c}}=0$. Finally, \eqref{3.3} follows from \eqref{eq:oc_1}
  and \eqref{4.3} from \eqref{eq:oc_2}, and since $\hat x_j\neq 0$ for
  all $j\in J_P$ and $\hat y_i\neq 0$ for all $i\in I_D$, we obtain, in
  particular, that
  \begin{align*}
    A_{J_{P}}^{\top}\hat y \odot \hat x_{J_{P}} &<0\\
    \text{and}\qquad -\sign(A^{I_{D}}\hat x - b_{I_{D}}) \odot \hat y_{I_{D}} &< 0.
  \end{align*}

  Keeping the above relations in mind, we proceed to show the
  equivalence of alternatives~\eqref{alt2} and~\eqref{alt3}: 

  Suppose that alternative~\eqref{alt2} is \emph{not} true, i.e.,
  there exists a~$d$ that satisfies \eqref{2.1}--\eqref{2.4}. As $d$
  fulfills \eqref{2.1} and \eqref{2.4}, we get that for each $t>0$,
  $(\hat x + td, t)$ fulfills \eqref{3.1} and \eqref{3.4},
  respectively. From \eqref{2.2} we obtain the existence of a
  $t_{1}>0$ such that $(\hat x+td, t)$ satisfies \eqref{3.2} for all
  $0\leq t\leq t_{1}$, and because of \eqref{2.3}, there exists a
  $t_{2}>0$ such that $(\hat x+td, t)$ fulfills \eqref{3.3} for all
  $0\leq t\leq t_{2}$. Consequently, we can choose $t=\min(t_{1},
  t_{2}, \hat \delta-\delta) > 0$ and have a corresponding feasible
  solution $(\hat x+td, t)$ of~\eqref{2.1}--\eqref{2.4}, which shows
  that alternative~\eqref{alt3} is not true either.

  Conversely, suppose that alternative~\eqref{alt3} is not true, i.e.,
  there exists a pair $(x, t)$ with $t > 0$ that satisfies
  \eqref{3.1}--\eqref{3.5}. We easily see that $d = (x - \hat x) / t$
  obeys \eqref{2.1}. Obviously, by construction, also \eqref{2.4}
  holds for~$d$. Moreover, it holds that
  \begin{align*}
    &\sign(A^{I_{P}\setminus I_{D}}\hat x - b_{I_{P}\setminus I_{D}}) \odot A^{I_{P}\setminus I_{D}}d\\
    = &\tfrac{1}{t}\sign(A^{I_{P}\setminus I_{D}}\hat x - b_{I_{P}\setminus I_{D}})\odot ([A^{I_{P}\setminus I_{D}}x - b_{I_{P}\setminus I_{D}}] - [A^{I_{P}\setminus I_{D}}\hat x - b_{I_{P}\setminus I_{D}}])\\
    = &\tfrac{1}{t}\sign(A^{I_{P}\setminus I_{D}}\hat x - b_{I_{P}\setminus I_{D}})\odot\sign(A^{I_{P}\setminus I_{D}} x - b_{I_{P}\setminus I_{D}})\odot\abs{A^{I_{P}\setminus I_{D}} x - b_{I_{P}\setminus I_{D}}}\\
    &- \tfrac{1}{t}\sign(A^{I_{P}\setminus I_{D}}\hat x - b_{I_{P}\setminus I_{D}})\odot\sign(A^{I_{P}\setminus I_{D}} \hat x - b_{I_{P}\setminus I_{D}})\odot\abs{A^{I_{P}\setminus I_{D}} \hat x - b_{I_{P}\setminus I_{D}}}\\
    \leq &\tfrac{\hat \delta- t}{t}\One - \tfrac{\hat \delta}{t}\One = -\One,
  \end{align*}
  so~$d$ satisfies \eqref{2.2} as well. Finally, \eqref{2.3} also
  holds true, since
  \begin{equation*}
    A^{\top}_{J_{D}\setminus J_{P}}\hat y\odot d_{J_{D}\setminus J_{P}} = \tfrac{1}{t}\underbrace{A^{\top}_{J_{D}\setminus J_{P}}\hat y\odot x_{J_{D}\setminus J_{P}}}_{\leq 0} - \tfrac{1}{t}A^{\top}_{J_{D}\setminus J_{P}}\hat y\odot \underbrace{\hat x_{J_{D}\setminus J_{P}}}_{= 0} \leq 0.
  \end{equation*}
  Thus, we conclude that alternative~\eqref{alt2} is indeed not true either.

  To complete the proof, it now suffices to show that
  alternatives~\eqref{alt1} and~\eqref{alt4} are equivalent. First,
  suppose that alternative~\eqref{alt1} is true, i.e., there exists
  an~$e$ that satisfies \eqref{1.0}--\eqref{1.4}. For arbitrary $s >
  0$, the vector $\hat y + se$ still obeys \eqref{4.1} and
  \eqref{4.4}, because of \eqref{1.1} and \eqref{1.4},
  respectively. Furthermore, \eqref{1.2} ensures that there exists an
  $s_{1}>0$ such that $\hat y + se$ still satisfies \eqref{4.2} for
  $0\leq s\leq s_{1}$, and \eqref{1.3} ensures the existence of an
  $s_{2}>0$ such that $\hat y + se$ obeys \eqref{4.3} for $0\leq s\leq
  s_{2}$. Thus, we can choose $s=\min(s_{1}, s_{2})$ and obtain that
  $\hat y + se$ satisfies \eqref{4.1}--\eqref{4.4}. Moreover,
  \eqref{1.0} shows that $-\sign(A\hat x - b)^{\top}(\hat y + se) <
  -\sign(A\hat x - b)^{\top}\hat y$ and it follows that $\hat y$ is
  \emph{not} the minimizer of \eqref{4.0}--\eqref{4.4} and thus, that
  alternative~\eqref{alt4} is true.

  Now, suppose conversely that alternative~\eqref{alt4} is true and
  that $y\neq \hat y$ is a minimizer of \eqref{4.0}--\eqref{4.4}. Then, $e \define y - \hat y$
  satisfies $-\sign(A\hat x - b)^{\top}(\hat y + e) < -\sign(A\hat x -
  b)^{\top}\hat y$, which shows that~$e$ obeys \eqref{1.0}. Moreover,
  \eqref{4.1}--\eqref{4.4} continue to hold for $\hat y + e$, which
  implies that $e$ satisfies \eqref{1.1}--\eqref{1.4} as well, and
  consequently, that alternative~\eqref{alt1} is true.\qed
\end{proof}

\subsection{$\ell_1$-\textsc{Houdini} Algorithm and Finite Termination}
\label{sec:algorithmandfinitetimetermination}

Theorem~\ref{theorem:alternatives} suggests the following algorithm: For a given $\delta_{k}>\delta$ and an optimal pair $(x^{k},y^{k})$ do: First update $y^{k+1}$ as a solution to \eqref{4.0}--\eqref{4.4} (with $\hat x=x^{k}$) and then find an updated $x^{k+1}$ and a $t^{k+1}>0$ as solution of \eqref{3.0}--\eqref{3.5} (with $\hat y=y^{k+1}$ and $\hat \delta = \delta^{k}$). In detail this is described in Algorithm~\ref{alg:homotopy_iteration}.

\SetAlFnt{\normalsize}
\DontPrintSemicolon
\begin{algorithm}[t]
 \KwIn{$A\in\R^{m\times n}$, $b\in\R^m$, $0 \leq \delta < \norm{b}_{\infty}$}
 \KwOut{solution $x^*$ to problem~\eqref{eq:p_delta}}   
 \BlankLine
 \tcp{Initialization:}
  $\delta^{0} \leftarrow \norm{b}_{\infty}$\;
  $x^{0} \leftarrow 0$\;

  $I_{P} \leftarrow \set{i : \abs{b_{i}} = \delta^{0}}$\;
  $J_{P} \leftarrow \emptyset$\;
  $k\leftarrow 0$\;
  \BlankLine
  \Repeat{$\delta^{k} = \delta$}{
    \tcp{Dual update:}
    $y^{k+1} \leftarrow$ solution of problem \eqref{4.0}--\eqref{4.4} with $\hat x = x^{k}$\;\label{algstep:dualupdate}
    $I_{D} \leftarrow \set{i: y^{k+1}_i\neq 0}$\;
    $J_{D} \leftarrow \set{j:\Abs{A^\top_j y^{k+1}}=1}$\;
    \BlankLine
    \tcp{Primal update:}
    $(x^{k+1}, t^{k+1}) \leftarrow$ sol. of problem \eqref{3.0}--\eqref{3.5} with $\hat y = y^{k+1}$ and $\hat \delta = \delta^{k}$\;\label{algstep:primalupdate}
    $\delta^{k+1} \leftarrow \delta^{k} - t^{k+1}$\;
    $I_{P} \leftarrow \set{i: \Abs{a_i^{\top}x^{k+1}-b_i}=\delta^{k+1}}$\;
    $J_{P} \leftarrow \set{j:x^{k+1}_j\neq 0}$\;
    \BlankLine
    $k \leftarrow k + 1$\;
 }
\BlankLine
\Return{$x^{*} = x^{k+1}$}
\caption{$\ell_1$-\textsc{Houdini}}
\label{alg:homotopy_iteration}
\end{algorithm}

To prove convergence of Algorithm~\ref{alg:homotopy_iteration} we start with a lemma:

\begin{lemma}
  \label{lemma:iterates}
  In each two consecutive iterations, Algorithm
  \ref{alg:homotopy_iteration} produces iterates $y^{k+1}\neq y^{k}$ and $x^{k+1}\neq x^{k}$. In particular, it holds that $t^{k+1}>0$ in each iteration.
\end{lemma}

\begin{proof}
  In the beginning, we have $x^{0} = 0$ and determine $y^{1}$ solving \eqref{4.0}--\eqref{4.4} with $\hat x = x^{0}$. By Theorem \ref{theorem:alternatives}, $(x^{0}, 0)$ is not an optimal solution to \eqref{3.0}--\eqref{3.5} with $\hat y = y^{1}$ and $\hat \delta = \delta^{0}$. It follows that $x^{1}\neq x^{0}$ and $t^{1} > 0$ after solving \eqref{3.0}--\eqref{3.5}.

Now suppose $k\geq 1$ and consider an iteration of Algorithm \ref{alg:homotopy_iteration} starting from an optimal pair $(x^{k}, y^{k})$ for (P$_{\delta^{k}}$) which is known from the previous iteration. First, we determine a new dual iterate $y^{k+1}$ by solving \eqref{4.0}--\eqref{4.4} with $\hat x = x^{k}$. From the previous primal update we know that $(x^{k}, \delta^{k-1} - \delta^{k})$ is a solution of \eqref{3.0}--\eqref{3.5} with $\hat \delta = \delta^{k-1}$ and $\hat y = y^{k}$. It follows that $(x^{k}, 0)$ is a solution of \eqref{3.0}--\eqref{3.5} with $\hat \delta = \delta^{k}$ and $\hat y = y^{k}$. In turn, Theorem \ref{theorem:alternatives} states that $y^{k}$ is not a solution of \eqref{4.0}--\eqref{4.4} with $\hat x = x^{k}$. By construction, $y^{k+1}$ is a solution of \eqref{4.0}--\eqref{4.4} with $\hat x = x^{k}$ and consequently $y^{k+1}\neq y^{k}$. For the same reason, Theorem \ref{theorem:alternatives} states that $(x^{k}, 0)$ is (although feasible) not a solution of \eqref{3.0}--\eqref{3.5} with $\hat y = y^{k+1}$ and $\hat \delta = \delta^{k}$. Since $(x^{k+1}, t^{k+1})$ is exactly such a solution, it follows that $t^{k+1} > 0$ and $x^{k+1}\neq x^{k}$.\qed
\end{proof}

Certainly, Lemma \ref{lemma:iterates} does not yet prove convergence
of Algorithm \ref{alg:homotopy_iteration}. Nevertheless, we see that
each iteration contributes at least a small approach towards a
solution of (\ref{eq:p_delta}).

\begin{theorem}
  \label{theorem:convergence}
  Algorithm \ref{alg:homotopy_iteration} terminates after a finite
  number of iterations and returns an optimal solution of
  (\ref{eq:p_delta}).
\end{theorem}
\begin{proof}
  The number of possible support sets $J_{P}$, active sets $I_{P}$,
  associated sign patterns and combinations thereof is finite. Suppose
  that for $k<\ell$ Algorithm~\ref{alg:homotopy_iteration} produces
  $J_{P}:=J_{P}^{k} = J_{P}^{\ell}$, $I_{P}:=I_{P}^{k} =
  I_{P}^{\ell}$, $\sign(x_{J_{P}}^{k}) = \sign(x_{J_{P}}^{\ell})$ and
  $\sign(Ax^{k} - b) = \sign(Ax^{\ell} - b)$. According to \eqref{4.0}--\eqref{4.4}
  we obtain that also $y^{k+1} = y^{\ell +1}$. It follows that the primal
  update steps \eqref{3.0}--\eqref{3.5} to find $x^{k+1}$ and $x^{\ell+1}$ are
  equal except that we have $\hat \delta = \delta^{k}$ in the
  first case and $\hat \delta = \delta^{\ell}$ in the second,
  where $\delta^{k} > \delta^{\ell}$ by Lemma \ref{lemma:iterates}. Since $\hat \delta$ is a constant, it is equivalent to rewrite \eqref{3.0} as $t - \hat \delta$. The substitution $\tilde \delta := \hat \delta - t$ in \eqref{3.0}--\eqref{3.2} and \eqref{3.5} then reveals that the update problems for $x^{k+1}$ and $x^{\ell + 1}$ indeed have an identical reformulation. Hence, we obtain the same optimal value for $\tilde \delta$ in both cases, which shows that $\delta^{k+1} = \delta^{\ell + 1}$ and contradicts Lemma
  \ref{lemma:iterates} since $k<\ell$. Thus, Algorithm
  \ref{alg:homotopy_iteration} terminates after a finite number of
  iterations with an optimal solution.\qed
\end{proof}

\section{Practical Considerations}
\label{sec:practicalconsiderations}

As mentioned earlier, one may in principle use an arbitrary LP solver to tackle the update problems in $\ell_1$-\textsc{Houdini}. However, due to their special structure, we found active-set strategies to be particularly efficient for these LPs. In the following, we give the details of our approach; the numerical experiments in Section~\ref{sec:applicationsandexamples} will later demonstrate the efficiency of our corresponding implementation.

\subsection{Active-Set Method for the Primal Update}
\label{sec:activesetprimal}

Finding a new primal iterate $x^{k+1}$ and the related decrease $t^{k+1}$ of the homotopy parameter in Step \ref{algstep:primalupdate} of Algorithm \ref{alg:homotopy_iteration} gives rise to the linear program
\begin{subequations}
  \begin{alignat}{5}
    \max_{(x_{J_{D}},t) \in \R^{\abs{J_{D}}}\times\R} &&\begin{pmatrix}0\\1\end{pmatrix}^{\top} &\begin{pmatrix}x_{J_{D}}\\t\end{pmatrix}\label{5.0}\\
    \st &&\quad \begin{bmatrix}A^{I_{D}}_{J_{D}} &\sign(y^{k+1}_{I_{D}})\end{bmatrix} &\begin{pmatrix}x_{J_{D}}\\t\end{pmatrix} &\quad = &\quad \delta^{k}\sign(y_{I_{D}}^{k+1}) + b_{I_{D}}\label{5.1}\\
    &&\quad \begin{bmatrix}A^{I_{D}^{c}}_{J_{D}} &\One\\ -A^{I_{D}^{c}}_{J_{D}} &\One\\ 0 &1\end{bmatrix} &\begin{pmatrix}x_{J_{D}}\\t\end{pmatrix} &\quad \leq &\quad \begin{pmatrix}\delta^{k}\One + b_{I_{D}^{c}}\\\delta^{k}\One - b_{I_{D}^{c}}\\\delta^{k} - \delta\end{pmatrix}\label{5.2}\\
    &&\begin{bmatrix}\diag(A^{\top}_{J_{D}}y^{k+1}) &0\\ 0 &-1\end{bmatrix} &\begin{pmatrix}x_{J_{D}}\\t\end{pmatrix} &\quad \leq &\quad 0\label{5.3}.
  \end{alignat}
  \label{eq:primalactivesetproblem}%
\end{subequations}
In this section, we introduce an active-set method in order to solve problem \eqref{5.0}-\eqref{5.3}. The idea for our approach bases upon the active-set method for quadratic programs illustrated, e.g., in \cite{NocedalWright2006}. We adapt the method to the special type of linear programs that we are faced with. We refer to Appendix \ref{sec:lp_active_set_method} for the general procedure and to Table \ref{tab:primalupdate} for supplementary details about the implementation of (\ref{eq:primalactivesetproblem}).

\subsubsection{Initialization}
\label{sec:activesetprimalinitialization}

We observe that the point $(x^{k}_{J_{D}}, 0)$ is feasible since $(x^{k}, y^{k+1})$ is an optimal pair for P$_{\delta^{k}}$. We set $\ell = 0$ and choose our starting point $(\xi^{\ell}_{J_{D}}, \tau^{\ell}) = (x^{k}_{J_{D}}, 0)$ accordingly. Regarding \eqref{5.2}, we see that the subset of active constraints at the starting point $(x^{k}_{J_{D}}, 0)$ corresponds to ${\cal A} = I_{P}\setminus I_{D}$ with either positive or negative sign. The initial support is exactly ${\cal S} = J_{P}$.

The variable $t$ represents the decrease of the homotopy parameter starting from $\delta^{k}$. Although the associated iterate is initially zero, $t$ joins the support once we have performed a step towards an ascent direction. Since each constructed direction is an ascent direction, $t$ does not leave the support afterwards. Consequently, we have ${\cal S} = J_{P} \cup \set{t}$.

The constraint $t\leq \delta^{k} - \delta$ is neither active in the beginning nor will it be so unless we have found an optimal solution of our original problem (\ref{eq:p_delta}).

\subsubsection{Ascent Directions and Blocking Constraints}
\label{sec:activesetprimaldescentdirectionsandblockingconstraints}

In order to find an ascent direction preserving ${\cal A}$ and ${\cal S}$, we fix $d_{J_{D}\setminus J_{P}} = 0$ and $d_{t} = 1$ and seek for a solution of the linear system
\begin{equation}
  \label{eq:primalactivesetdescentdirection}
  A_{J_{P}}^{I_{P}}d_{J_{P}} = -\sign(A^{I_{P}}\xi^{\ell} - b_{I_{P}}).
\end{equation}
If a solution of (\ref{eq:primalactivesetdescentdirection}) exists, the largest step size $\alpha$ preserving feasibility is
\begin{equation}
  \label{eq:primalalpha}
  \alpha = \min\set{\alpha_{{\cal A}},\, \alpha_{{\cal S}},\, \delta^{k} - \tau^{\ell} - \delta},
\end{equation}
wherein
\begin{equation}
  \label{eq:primalactivesetalphacon}
  \alpha_{{\cal A}} = \min \left\{\min_{i\in I_{P}^{c}\atop a_{i}^{\top}d>-1} \frac{\delta^{k} - \tau^{\ell} - a_{i}^{\top}\xi^{\ell} + b_{i}}{a_{i}^{\top}d + 1}, \min_{i\in I_{P}^{c}\atop a_{i}^{\top}d<1} \frac{\delta^{k} - \tau^{\ell} + a_{i}^{\top}\xi^{\ell} - b_{i}}{-a_{i}^{\top}d + 1}\right\}
\end{equation}
and
\begin{equation}
  \label{eq:primalactivesetalphasup}
  \alpha_{{\cal S}} = \min_{j\in J_{P}\atop A_{j}^{\top}y^{k+1}\cdot d_{j} > 0} -\frac{\xi^{\ell}_{j}}{d_{j}}.
\end{equation}
The new iterates are then
\begin{equation}
  \xi^{\ell + 1} = \xi^{\ell} + \alpha d\qquad \text{and} \qquad \tau^{\ell + 1} = \tau^{\ell} + \alpha.
\end{equation}
In case $\alpha = \delta^{k} - \tau^{\ell} - \delta$, we stop thereafter since $x^{*} = \xi^{\ell + 1}$ is an optimal solution of (\ref{eq:p_delta}). Otherwise, we finally update
\begin{equation}
  \label{eq:primalactivesetblockingconstraints}
  \begin{aligned}
    I_{P} &= I_{P} \cup \set{i\in I_{P}^{c} : \abs{A^{i}\xi^{\ell + 1} - b_{i}} = \delta^{k}- \tau^{\ell + 1}}\\
    J_{P} &= J_{P} \setminus \set{j\in J_{P} : \abs{\xi^{\ell+1}} = 0}
  \end{aligned}
\end{equation}
which corresponds to an update of ${\cal A} = I_{P}\setminus I_{D}$ and ${\cal S} = J_{P}$.

\subsubsection{Lagrange Multipliers}
\label{sec:primalactivesetlagrangemultipliers}

If a solution of (\ref{eq:primalactivesetdescentdirection}) does not exist, zero is an optimal solution of
\begin{equation*}
    \max_{(d_{J_{P}}, d_{t}) \in \R^{\abs{J_{P}}}\times \R} \ \begin{pmatrix}0\\1\end{pmatrix}^{\top} \begin{pmatrix}d_{J_{P}}\\d_{t}\end{pmatrix} \quad \mathrm{s.t.} \ \begin{bmatrix}A^{I_{P}}_{J_{P}} &\sign(A^{I_{P}}\xi^{\ell} - b_{I_{P}})\end{bmatrix} \begin{pmatrix}d_{J_{P}}\\ d_{t}\end{pmatrix} = 0
\end{equation*}
and the associated KKT conditions show that there exists $\hat e_{I_{P}}$ satisfying
\begin{equation}
  \label{eq:primalactivesetmultipliers1}
  \begin{aligned}
    (A^{I_{P}}_{J_{P}})^{\top} \hat e_{I_{P}} &= 0\\
    \sign(A^{I_{P}}\xi^{\ell} - b_{I_{P}})^{\top} \hat e_{I_{P}} &= 1.
  \end{aligned}
\end{equation}
Building on that, we set
\begin{align}
  \label{eq:primalactivesetmultipliers2}
  \mu_{I_{P}\setminus I_{D}} &= \sign(A^{I_{P}\setminus I_{D}}\xi^{\ell} - b_{I_{P}\setminus I_{D}})\odot \hat e_{I_{P}\setminus I_{D}}\\
  \label{eq:primalactivesetmultipliers3}
  \nu_{J_{D}\setminus J_{P}} &= -(A_{J_{P}\setminus J_{D}}^{\top}y^{k+1}) \odot (A_{J_{D}\setminus J_{P}}^{I_{P}})^{\top}\hat e_{I_{P}}.
\end{align}
We can consider $\mu_{I_{P}\setminus I_{D}}$ and $\nu_{J_{D}\setminus J_{P}}$ as Lagrange multipliers associated with the KKT conditions for (\ref{eq:primalactivesetproblem}). In particular, $\mu_{I_{P}\setminus I_{D}}$ corresponds to the set ${\cal A}$ of active constraints in (\ref{5.2}) and $\nu_{J_{D}\setminus J_{P}}$ to ${\cal S}^{c}$, i.e., the active constraints in (\ref{5.3}). Although differently motivated, the multipliers (\ref{eq:primalactivesetmultipliers2}) and (\ref{eq:primalactivesetmultipliers3}) are exactly what we get if we determine $\mu_{{\cal A}}$ and $\nu_{{\cal S}^{c}}$ according to Appendix \ref{sec:lagrange_multipliers}.

In case $\mu_{I_{P}\setminus I_{D}}\geq 0$ and $\nu_{J_{D}\setminus J_{P}}\geq 0$, the current iterate $\xi_{J_{D}}^{\ell}$ is optimal. Else, we pick $i\in I_{P}\setminus I_{D}$ with $\mu_{i} < 0$ or $j\in J_{D}\setminus J_{P}$ with $\nu_{j} < 0$ and update $I_{P} = I_{P}\setminus\set{i}$ or $J_{P} = J_{P} \cup \set{j}$, respectively. This corresponds to an update of ${\cal A}$ and ${\cal S}$, respectively. 

\subsection{Active-Set Method for the Dual Update}
\label{sec:activesetdual}

Finding a new dual iterate $y^{k+1}$ in Step \ref{algstep:dualupdate} of Algorithm \ref{alg:homotopy_iteration} gives rise to to the linear program
\begin{subequations}
  \begin{alignat}{5}
    \min_{y_{I_{P}} \in \R^{\abs{I_{P}}}} &&\quad -\sign(A^{I_{P}}x^{k} - b_{I_{P}})^{\top} &y_{I_{P}}\label{8.0}\\
    \st &&\quad (-A^{I_{P}}_{J_{P}})^{\top} &y_{I_{P}} &\quad = &\quad \sign(x^{k}_{J_{P}})\label{8.1}\\
    &&\quad \begin{bmatrix}(A^{I_{P}}_{J_{P}^{c}})^{\top}\\ (-A^{I_{P}}_{J_{P}^{c}})^{\top}\end{bmatrix} &y_{I_{P}} &\quad \geq &\quad -\One\label{8.2}\\
    &&\diag(\sign(A^{I_{P}}x^{k} - b_{I_{P}})) &y_{I_{P}} &\quad \geq &\quad 0\label{8.3}.
  \end{alignat}
  \label{eq:dualactivesetproblem}%
\end{subequations}
Analogous to the primal case, we use our results from Appendix \ref{sec:lp_active_set_method} to develop an active-set method for problem \eqref{8.0}--\eqref{8.3}. We refer to Table \ref{tab:dualupdate} for additional information on the implementation of (\ref{eq:dualactivesetproblem}).

\subsubsection{Initialization}
\label{sec:dualactivesetinitialization}

In the beginning, $y^{k}_{I_{P}}$ is feasible since $(x^{k}, y^{k})$ is an optimal pair. We set $\ell = 0$ and choose $\psi^{\ell}_{I_{P}} = y^{k}_{I_{P}}$ as our starting point. In view of \eqref{8.2}, the set of active constraints at $y^{k}_{I_{P}}$ corresponds to ${\cal A} = J_{D}\setminus J_{P}$ with either positive or negative sign and the initial support is ${\cal S} = I_{D}$.

\subsubsection{Descent Direction and Blocking Constraints}
\label{sec:dualactivesetsubproblem}

We seek for a descent direction preserving ${\cal A}$ and ${\cal S}$ by solving
\begin{equation}
  \label{eq:dualactivesetdirection}
  \begin{aligned}
    (A^{I_{D}}_{J_{D}})^{\top} e_{I_{D}} &= 0\\
    \sign(A^{I_{D}}x^{k} - b_{I_{D}})^{\top} e_{I_{D}} &=1.
  \end{aligned}
\end{equation}
If such a direction exists, the largest step size preserving feasibility is
\begin{equation}
  \label{eq:dualalpha}
  \alpha = \min\set{\alpha_{{\cal A}},\, \alpha_{{\cal S}}}.
\end{equation}
Here,
\begin{equation}
  \label{eq:dualactivesetalphacon}
  \alpha_{{\cal A}} = \min \left\{ \min_{j\in J_{D}^{c}\atop A_{j}^{\top}e < 0} \frac{1 + A_{j}^{\top}\psi^{\ell}}{-A_{j}^{\top}e}, \min_{j\in J_{D}^{c}\atop A_{j}^{\top}e > 0} \frac{1 - A_{j}^{\top}\psi^{\ell}}{A_{j}^{\top}e}\right\}
\end{equation}
and
\begin{equation}
  \label{eq:dualactivesetalphasup}
  \alpha_{{\cal S}} = \min_{i\in I_{D}\atop \sign(a_{i}^{\top}x^{k} - b_{i})e_{i} < 0} -\frac{\psi^{\ell}_{i}}{e_{i}}.
\end{equation}
The new iterate is $\psi^{\ell + 1} = \psi^{\ell} + \alpha e$. Finally, we need to update
\begin{equation}
  \label{eq:dualactivesetblockingconstraints}
  \begin{aligned}
    I_{D} &= I_{D} \setminus \set{i\in I_{D} : \psi^{\ell + 1} = 0}\\
    J_{D} &= J_{D} \cup \set{j\in J_{D}^{c} : \abs{A_{j}^{\top}\psi^{\ell + 1}} = 1}
  \end{aligned}
\end{equation}
which corresponds to an upate of ${\cal A} = J_{D}\setminus J_{P}$ and ${\cal S} = I_{D}$.

\subsubsection{Lagrange Multipliers}
\label{sec:dualactivesetlagrangemultipliers}

If a solution of (\ref{eq:dualactivesetdirection}) does not exist, then zero is an optimal solution of
\begin{equation*}
  \min_{e_{I_{D}}\in\R^{\abs{I_{D}}}} \ -\sign(A^{I_{D}}x^{k} - b_{I_{D}})^{\top}e_{I_{D}} \quad \mathrm{s.t.} \ (A^{I_{D}}_{J_{D}})^{\top}e_{I_{D}} = 0.
\end{equation*}
Analogous to above, KKT conditions ensure that there exists $\hat d_{J_{D}}$ such that
\begin{equation}
  \label{eq:dualactivesetmultipliers1}
  A_{J_{D}}^{I_{D}} \hat d_{J_{D}} = -\sign(A^{I_{D}}x^{k} - b_{I_{D}})
\end{equation}
and we obtain Lagrange multipliers for (\ref{eq:dualactivesetproblem}) by setting 
\begin{align}
  \mu_{J_{D}\setminus J_{P}} &= -(A_{J_{D}\setminus J_{P}}^{\top}\psi^{\ell})\odot\hat d_{J_{D}\setminus J_{P}}\label{eq:dualactivesetmultipliers2}\\
  \nu_{I_{P}\setminus I_{D}} &= -\sign(A^{I_{P}\setminus I_{D}}x^{k} - b_{I_{P}\setminus I_{D}}) \odot A^{I_{P}\setminus I_{D}}_{J_{D}}\hat d_{J_{D}} - \One.\label{eq:dualactivesetmultipliers3}
\end{align}
Here, $\mu_{J_{D}\setminus J_{P}}$ corresponds to the set ${\cal A}$ of active constraints in (\ref{8.2}) and $\nu_{I_{P}\setminus I_{D}}$ correpsonds to ${\cal S}^{c}$, i.e., the set of active constraints in (\ref{8.3}). These multipliers are equal to those we obtain according to Appendix \ref{sec:lagrange_multipliers}.

In case $\mu_{J_{D}\setminus J_{P}} \geq 0$ and $\nu_{I_{P}\setminus I_{D}} \geq 0$, the current iterate $\psi_{I_{P}}^{\ell}$ is optimal. Otherwise, we can find $j\in J_{D}\setminus J_{P}$ with $\mu_{j} < 0$ or $i\in I_{P}\setminus I_{D}$ with $\nu_{i} < 0$ and update $J_{D} = J_{D} \setminus \set{j}$ or $I_{D} = I_{D} \setminus \set{i}$, repsectively. 

\subsection{Links Between Primal and Dual Active-Set Method}
\label{sec:linksbetweenprimalanddualactivesetmethod}

In the following, we establish a close connection between the methods discussed in Sections \ref{sec:activesetprimal} and \ref{sec:activesetdual}. This natural link will turn out to be enormously useful in terms of computational efficiency.
\medskip

In the context of Section \ref{sec:primalactivesetlagrangemultipliers}, suppose that we have found $\hat e_{I_{P}}$ satisfying equations (\ref{eq:primalactivesetmultipliers1}) such that the associated Lagrange multipliers $\mu_{I_{P}\setminus I_{D}}$ and $\nu_{J_{D}\setminus J_{P}}$ are throughout non-negative. In that situation, we have found an optimal solution of the primal subproblem \eqref{eq:primalactivesetproblem} and proceed to the dual subproblem \eqref{eq:dualactivesetproblem}. Therein, we would first attempt to find a direction $e_{I_{D}}$ satisfying (\ref{eq:dualactivesetdirection}). Can this ever be successful?
\medskip

Let us recall the situation at the end of the previous dual update. In fact, we did not find a direction satisfying (\ref{eq:dualactivesetdirection}) and afterwards found that our current iterate was already optimal. Since then, the sets $I_{D}$ and $J_{D}$ did not change. Hence, it would be pointless to search a solution of (\ref{eq:dualactivesetdirection}) as a first step of the active-set method for the dual update.
\medskip

As we have argued so far, we would continue by adapting the sets $I_{D}$ and $J_{D}$ invoking Lagrange multipliers according to (\ref{eq:dualactivesetmultipliers1})--(\ref{eq:dualactivesetmultipliers3}). But there is a remedy. A comparison of what we have and what we seek for, $\hat e_{I_{P}}$ and $e_{I_{D}}$, respectively, reveals the follwing:
\begin{alignat*}{5}
  (A_{J_{P}}^{I_{P}})^{\top}\hat e_{I_{P}} &= 0 \qquad \qquad &(A_{J_{D}}^{I_{D}})^{\top}e_{I_{D}} &= 0\\
  \sign(A^{I_{P}}x^{k} - b_{I_{P}})^{\top}\hat e_{I_{P}} &= 1 \qquad \qquad &\sign(A^{I_{D}}x^{k} - b_{I_{D}})^{\top}e_{I_{D}} &= 1.
\end{alignat*}
The crucial idea is now to perform the updates
\begin{equation}
  \label{eq:dualactivesetaprioriupdate}
  \begin{aligned}
    I_{D} &= I_{D} \cup \{i\in I_{P}\setminus I_{D} : \hat e_{i} \neq 0\}\\
    J_{D} &= J_{D} \setminus \{j\in J_{D}\setminus J_{P} : (A_{j}^{I_{P}})^{\top}\hat e_{I_{P}} \neq 0\}.
  \end{aligned}
\end{equation}
After that, $e_{I_{D}} = \hat e_{I_{D}}$ will do exactly what we need.
\medskip

The fact that the Lagrange multipliers associated with $\hat e_{I_{P}}$ are non-negative throughout shows that a non-trivial step $y^{k} + \alpha\hat e$ maintains primal-dual optimality. For $i \in I_{P}\setminus I_{D}$ with $\hat e_{i} \neq 0$, it holds that $\sign(a_{i}^{\top}x^{k} - b_{i})\hat e_{i} > 0$, which shows that a step in direction $\hat e$ provides the dual variable with the desired sign. Further, it holds for $j\in J_{D}\setminus J_{P}$ with $A_{j}^{\top}\hat e \neq 0$ that $A_{j}^{\top}y^{k} \cdot A_{j}^{\top}\hat e < 0$, which shows that a step in direction $\hat e$ forces the respective dual constraint to become inactive while maintaining feasibility.
\medskip

It is not at all surprising that an analogous approach works in the beginning of the primal update. Suppose that we have $\hat d_{J_{D}}$ according to \eqref{eq:dualactivesetmultipliers1} at hand and the associated Lagrange multipliers are non-negative. We compare $\hat d_{J_{D}}$ to the sought after direction $d_{J_{P}}$:
\begin{equation*}
  A_{J_{D}}^{I_{D}}\hat d_{J_{D}} = -\sign(A^{I_{D}}x^{k}-b_{I_{D}}) \quad \qquad A_{J_{P}}^{I_{P}}d_{J_{P}} = -\sign(A^{I_{P}}x^{k} - b_{I_{P}}).
\end{equation*}
Analogous to above, we perform the update
\begin{equation}
  \label{eq:primalactivesetaprioriupdate}
  \begin{aligned}
    J_{P} &= J_{P} \cup \{j\in J_{D}\setminus J_{P} : \hat d_{j} \neq 0\}\\
    I_{P} &= I_{P} \setminus \{i\in I_{P}\setminus I_{D} : a_{i}^{\top}\hat d\neq -\sign(a_{i}^{\top}x^{k} - b_{i})\},
  \end{aligned}
\end{equation}
whereafter $d_{J_{P}} = \hat d_{J_{P}}$ does the job.
\medskip

By non-negativity of the Lagrange multipliers associated with $\hat d_{J_{D}}$, it can be shown that a non-trivial step $x^{k} + \alpha \hat d$ maintains primal-dual optimality: For $j\in J_{D}\setminus J_{P}$ with $\hat d_{j}\neq 0$ it holds that $-A_{j}^{\top}y^{k}\cdot\hat d_{j} > 0$. Further, each $i\in I_{P}\setminus I_{D}$ with $a_{i}^{\top}\hat d\neq -\sign(a_{i}^{\top}x^{k} - b_{i})$ satisfies $a_{i}^{\top}\hat d\cdot\sign(a_{i}^{\top}x^{k} - b_{i}) < -1$.

\section{Applications and Examples}
\label{sec:applicationsandexamples}

Before we come to a numerical evaluation of the algorithm, a typical run of $\ell_{1}$-\textsc{Houdini} on a small problem is shown in Figure~\ref{fig:solution_path}.
We observe that the solution path does not need to show any particular monotonicity; other examples exhibit even more tangled solution paths with multiple variables entering or leaving the support or dense clusters of break points of $\delta^{k}$ at various values.

\begin{figure}
  \centering
  \includegraphics[scale = .65]{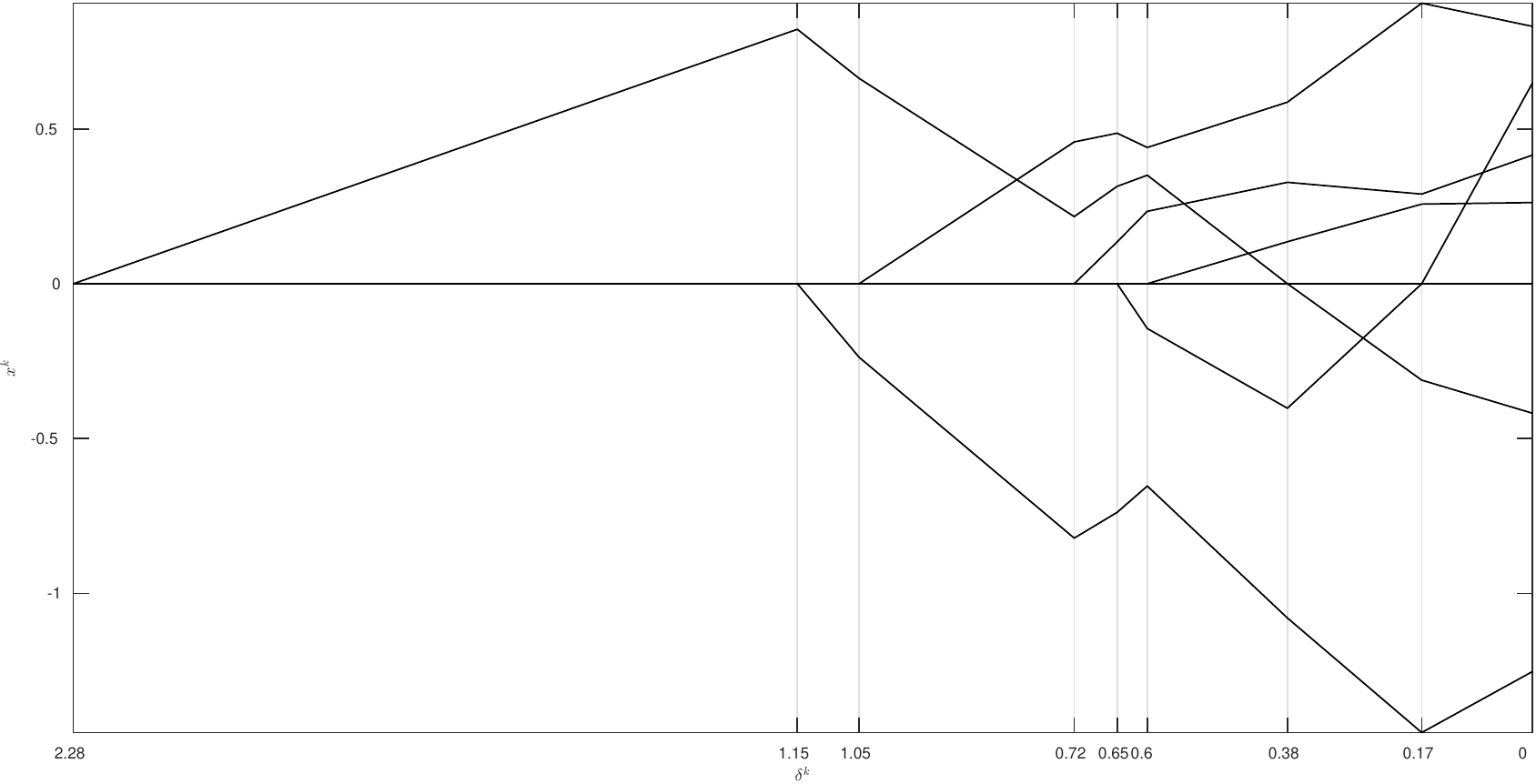}
  \caption{Examplary run of $\ell_{1}$-\textsc{Houdini} (using active set) with $A\in\R^{6\times 12}$ and $b\in\R^{6}$ randomly generated and $\delta = 0$. The algorithm needed 9 iterations to solve the problem. Horizontal labels display the value of the homotopy parameter $\delta^{k}$ after each iteration. The plots represent the solution paths of $x^{k}_{j}$ for $j = 1,\dots,12$. The optimal solution has $6$ nonzero entries.}
  \label{fig:solution_path}
\end{figure}

We compare our homotopy
method for \eqref{eq:p_delta} with the state-of-the-art commercial LP solver Gurobi applied to the LP reformulation
\[
  \min\,\mathds{1}^\top x^++\mathds{1}^\top x^-\quad\text{s.t.}\quad -\delta\cdot\mathds{1}\leq Ax^+-Ax^--b\leq\delta\cdot\mathds{1},\,x^+\geq 0,\,x^-\geq 0
\]
(note that this formulation is equivalent to the one stated in
Section~\ref{subsec:PSM}, which contains slack variables).  We
experiment with two variants of our $\ell_1$-\textsc{Houdini} algorithm: In
one, we use the specialized active-set methods described in
Section~\ref{sec:practicalconsiderations}, and in
the other, we employ the same LP solver for our primal and dual update
subproblems that we compare against, i.e., with which we solve the
above LP reformulation of~\eqref{eq:p_delta}.

Our $\ell_1$-\textsc{Houdini} is implemented in \textsc{Matlab}. From the same wrapper
code to read instance data, we call either $\ell_1$-\textsc{Houdini} to solve
for the entire homotopy path, or call Gurobi (via its \textsc{Matlab} interface). 

The test instances are constructed from the ``L1-Testset'' described
in~\cite{LorenzPfetschTillmann2015}. This test set (available online
via the last author's or the SPEAR project homepage) contains over 500
instances $A$, $\bar{x}$ and $b=A\bar{x}$ for the pure Basis Pursuit
problem~\eqref{eq:bp} such that $\bar{x}$ is the known unique optimal
solution; all solutions in the test set are relatively sparse and have
varying dynamic ranges. Based on the following result, we can (for a
given $\delta$) construct new vectors $\hat b$ such that $\bar{x}$ is
optimal for the instance of~\eqref{eq:p_delta} specified by $A$, $\hat b$
and $\delta$.
\begin{lemma}\label{lem:constructPdeltaSol}
  Let $\bar{x}$ be an optimal solution of~\eqref{eq:bp} with given $A$
  and $b=A\bar{x}$. Then, $\bar{x}$ is an optimal solution
  of~\eqref{eq:p_delta} with the same $A$ and a measurement vector
  $b=\hat{b}$ if and only if there exists $\bar{y}$ such that 
  \[
  -A^\top \bar{y}\in\Sign(\bar{x})\quad\text{and}\quad\hat{b}\in A\bar{x}-\delta\Sign(\bar{y}).
  \]
\end{lemma}
\begin{proof}
  Optimality of $\bar{x}$ for~\eqref{eq:bp} with $A$ and $b=A\bar{x}$
  is characterized by the existence of a vector $\bar{y}$ such that
  $-A^\top\bar{y}\in\Sign(\bar{x})$, see,
  e.g.,~\cite{LorenzPfetschTillmann2015}. Choosing $\hat{b}\in
  A\bar{x}-\delta\Sign(\bar{y})$, we obtain that additionally,
  $A\bar{x}\in\hat{b}+\delta\Sign(\bar{y})$. The claim now follows
  immediately from~\eqref{eq:oc_1} and~\eqref{eq:oc_2}.\qed
\end{proof}

To use Lemma~\ref{lem:constructPdeltaSol} to construct instances
for~\eqref{eq:p_delta}, note that in addition to $A$ and $\bar{x}$, we
also need an optimal dual certificate $\bar{y}$ for the
associated~\eqref{eq:bp} problem. For the L1-Testset instances, a
detailed description of how this can be computed is provided
in~\cite[Sections~4 and 5 (particularly,
Theorem~5.1)]{LorenzPfetschTillmann2015}; in short, we can either obtain $\bar{y}$ with a closed-form expression or apply alternating projections onto $\Sign(\bar{x})$ and the image space of $A^\top$. The vectors $\hat{b}$ are then constructed as $A\bar{x}-\delta\sign(\bar{y})$. 
For the present experiments, we randomly choose two instances for each
of the matrix sizes $512\times\{1025,1536,2048,4096\}$ and
$1024\times\{2048,3072,4096,8192\}$
(cf.~\cite[Table~II]{LorenzPfetschTillmann2015})---one in which
$\bar{x}$ has nonzero entries of high dynamic range, and one with low
dynamic range. This way, we end up with 16 instances, which
we will identify by their L1-Testset number (the instance details can
be found in the table accompanying the test instance download
package). The $\delta$-values were chosen uniformly at random from the interval $[0.1,5]$ for each instance.
Moreover, since we observed that the $\bar{y}$ constructed in the above-mentioned ways are fully dense (which, by complementary slackness, implies that the primal active sets in the respective optimal solutions are also as large as possible), we computed a second set of $\hat{b}$-vectors using other dual certificates that were computed, aiming at sparsity, by solving problems of the form
\[
\min_{y\in\R^m} \norm{y}_1\quad\st\quad -A^\top y\in\Sign(\bar{x}).
\]
Thus, we have 32 instances in total, with pairs sharing the same instance number, $A$, $\delta$ and optimal solution $\bar{x}$ but having different measurement vectors~$\hat{b}$.
(Regarding instance constructions for basis pursuit and related problems in general, it is worth mentioning that the above $\ell_1$-minimization problem to compute dual certificates can be solved very efficiently via its straightforward LP reformulation, even for large-scale data where an alternating projection approach may no longer work or take an unreasonably long time.)

\begin{table}[tb]
  \centering
  \begin{tabularx}{\textwidth}{@{}r@{\qquad}r@{\qquad}r@{\qquad}r@{\qquad}r@{\qquad}r@{\qquad}r@{\qquad}X@{}}\toprule
  inst. no. & $m\times n$ & $\delta$ & $\abs{{\cal S}}$ & $\abs{{\cal A}}$ & \multicolumn{2}{c}{time $\ell_1$-HOUDINI} & time Gurobi\\
            & & & &                     & (active set) & (Gurobi) & \\\midrule
  7         & $512\times 1024$  & 4.09 & 34 & 512  & 0.48  & 2.44   & 0.47 \\ 
            &                   &      &    & 72   & -     & 2.46   & 0.47 \\
  485       & $512\times 1024$  & 4.54 & 51 & 512  & 1.68  & 98.51  & 1.31 \\ 
            &                   &      &    & 96   & 1.01  & -      & 1.12 \\
  25        & $512\times 1536$  & 0.72 & 14 & 512  & 0.20  & 3.46   & 0.82 \\ 
            &                   &      &    & 31   & 0.19  & 3.50   & 0.81 \\
  319       & $512\times 1536$  & 4.58 & 22 & 512  & 0.38  & 15.16  & 1.70 \\ 
            &                   &      &    & 43   & 0.24  & 9.64   & 1.53 \\
  228       & $512\times 2048$  & 3.20 & 51 & 512  & 5.09  & -      & 1.10 \\ 
            &                   &      &    & 141  & 3.08  & -      & 0.95 \\
  338       & $512\times 2048$  & 0.58 & 20 & 512  & 0.70  & -      & 1.93 \\ 
            &                   &      &    & 45   & 0.36  & 15.19  & 1.43 \\
  74        & $512\times 4096$  & 1.47 & 10 & 512  & 0.16  & 17.87  & 1.27 \\ 
            &                   &      &    & 38   & 0.11  & 1.00   & 1.22 \\
  347       & $512\times 2048$  & 2.78 & 10 & 512  & 0.10  & 8.18   & 1.25 \\ 
            &                   &      &    & 32   & 0.06  & 0.82   & 1.24 \\
  239       & $1024\times 2048$ & 4.79 & 84 & 1024 & 0.62  & 2.00   & 0.08 \\ 
            &                   &      &    & 148  & 0.60  & 1.86   & 0.07 \\
  357       & $1024\times 2048$ & 4.83 & 27 & 1024 & 1.63  & -      & 3.41 \\ 
            &                   &      &    & 55   & 0.65  & 36.92  & 2.73 \\
  99        & $1024\times 3072$ & 0.87 & 18 & 1024 & 0.71  & 19.02  & 3.40 \\ 
            &                   &      &    & 47   & 0.58  & 16.45  & 3.45 \\
  527       & $1024\times 3072$ & 4.86 & 99 & 1024 & 20.37 & -      & 1.75 \\ 
            &                   &      &    & 234  & 11.43 & -      & 1.54 \\
  263       & $1024\times 4096$ & 4.79 & 97 & 1024 & 30.76 & -      & 2.88 \\ 
            &                   &      &    & 245  & 22.23 & 411.83 & 2.55 \\
  416       & $1024\times 4096$ & 2.48 & 26 & 1024 & 1.89  & -      & 6.74 \\ 
            &                   &      &    & 60   & 1.01  & 47.69  & 3.93 \\
  148       & $1024\times 8192$ & 4.02 & 20 & 1024 & 1.13  & 21.97  & 4.90 \\ 
            &                   &      &    & 64   & 1.01  & 19.42  & 4.82 \\
  421       & $1024\times 8192$ & 0.80 & 9  & 1024 & 0.60  & -      & 4.92 \\ 
            &                   &      &    & 43   & 0.26  & -      & 4.82 \\\bottomrule
  \end{tabularx}
  \caption{Runtime comparison of $\ell_1$-\textsc{Houdini} against Gurobi.}
  \label{tab:HOUDINIvsGurobi}
\end{table}

The running time results of our experiments (conducted in
\textsc{Matlab}~2014a, using Gurobi~6.5.2, on Ubuntu with an Intel\textsuperscript{\textregistered} Core\textsuperscript{\texttrademark} i7-4550U CPU @ 1.50GHz $\times$ 4 processor) are summarized in Table~\ref{tab:HOUDINIvsGurobi}. 

In the majority of cases, we observed that $\ell_{1}$-\textsc{Houdini} using specialized active-set methods for the subproblems is considerably faster than $\ell_{1}$-\textsc{Houdini} using Gurobi (31 out of 32 instances) and even faster than Gurobi used as standalone LP solver (21 out of 32 instances). Another comparison suggests that Gurobi used as standalone solver is usually faster than $\ell_{1}$-\textsc{Houdini} using Gurobi for the subproblems (30 out of 32 instances). (Nevertheless, note that $\ell_{1}$-\textsc{Houdini} generates the entire solution path w.r.t. the homotopy parameter, whereas solving the LP formulation of (\ref{eq:p_delta}) solely yields a solution for the final parameter $\delta$.)

In particular, it seems beneficial to use $\ell_{1}$-\textsc{Houdini} when $\abs{{\cal S}}$ is small (i.e., when the optimal solution $x^{*}$ is relatively sparse). This is a natural feature of our method since the sparsity of the iterates has direct impact on the size of the arising subproblems. Analogously, the size of the primal active set $\abs{\cal A}$ directly affects the size of the subproblems. Our experiments show that solving the very same instace with smaller optimal active set (induced by a modified measurement vector $\hat b$) causes an average speedup of 33.9\% and 31.4\% using $\ell_{1}$-\textsc{Houdini} with active-set methods and Gurobi for the subproblems, respectively. In contrast, using Gurobi as standalone LP solver induces an average speedup of 10.4\%.

In additional experiments, we observed that $\ell_{1}$-\textsc{Houdini} is also competitive in the Basis Pursuit setting ($\delta\approx 0$). To that end, we compared our method with $\ell_{1}$-Homotopy and SPGL1, two of the fastest methods according to \cite{LorenzPfetschTillmann2015}. Finally, we performed testruns on some of the large-scale instances with sparse coefficient matrices from the L1-Testset, where $\ell_{1}$-\textsc{Houdini} was competitive as well and often considerably faster than Gurobi (even though Gurobi is tuned for sparse data). However, we need to mention that our current implementation (availabe on the first author's homepage) suffers numerical issues on particular instances of our testset, especially on those with sparse coefficient matrices.

\section{Extensions and Conclusion}
\label{sec:extensions}
Our algorithm can be extended straightforwardly to treat the more
general problem class
\begin{equation}\label{eq:gen_prob}
  \min_{x\in\R^n} \norm{x}_1\quad\st\quad \alpha\leq Ax-b\leq\beta,~Dx=d,
\end{equation}
assuming w.l.o.g. that $\alpha<\beta$ and that the feasible set is
nonempty. 

To that end, 
first observe that we can rewrite
\[
\alpha\leq Ax-b\leq\beta \quad\Leftrightarrow\quad \underbrace{\alpha-\tfrac{\alpha+\beta}{2}}_{=-\gamma}\leq Ax-\underbrace{\left(b+\tfrac{\alpha+\beta}{2}\right)}_{\eqqcolon\tilde{b}}\leq\underbrace{\beta-\tfrac{\alpha+\beta}{2}}_{\eqqcolon\gamma};
\]
since $\alpha<\beta$, $\gamma_i\neq 0$ for all~$i$, we can scale each
row by $\hat\delta/\gamma_i$ for an arbitarily chosen $\hat\delta>0$ and obtain 
\begin{align*}
&-\hat\delta G\gamma\leq G(Ax-\tilde{b})\leq\hat\delta G\gamma\\
\quad\Leftrightarrow\quad &-\hat\delta\One\leq GAx-G\tilde{b}\leq\hat\delta\One \quad\Leftrightarrow\quad \norm{GAx-G\tilde{b}}_\infty\leq\hat\delta,
\end{align*}
where $G=\diag(1/\gamma_1,\dots,1/\gamma_m)$. Thus, in the absence of
equality constraints $Dx=d$, \eqref{eq:gen_prob} can be recast into
the form~\eqref{eq:p_delta} directly. 

However, such an equality constraint is obviously equivalent to
requiring $\norm{Dx-d}_\infty\leq 0$. Therefore, we can extend the
homotopy treatment of problem~\eqref{eq:p_delta} (where we drive the
homotopy parameter down to the target $\delta$-value) to
\eqref{eq:gen_prob} by linking the homotopy parameter $\delta$ to the
bounds from both $\ell_\infty$-norm constraints derived from
\eqref{eq:gen_prob} and reducing it all the way to zero. For
$\delta=0$, the homotopy constraints
$\norm{GAx-G\tilde{b}}_\infty\leq\hat\delta+\delta$ and
$\norm{Dx-d}_\infty\leq\delta$ then correspond exactly to those
of~\eqref{eq:gen_prob}. Considering two $\ell_\infty$-norm constraints
simultaneously, and the offset $\hat\delta$ in one of them, leads to
minor simple modifications to the update subproblems in our algorithm;
we omit the straightforward details for brevity. Note that for
$\delta=\delta^0\define\max\{\norm{d}_\infty,\norm{b}_\infty-\hat\delta\}$,
$x=0$ is an optimal solution for the problem
\[
\min_{x\in\R^n} \norm{x}_1\quad\st\quad \norm{GAx-G\tilde{b}}_\infty\leq\hat\delta+\delta,~\norm{Dx-d}_\infty\leq\delta
\]
and thus provides the starting point for our method in the present
context.

Further generalizations are likely possibly. For instance, it should
be possible to modify the algorithm to treat one-sided bounds
($\alpha_i=-\infty$ or $\beta_i=+\infty$); then, in particular, the
case of nonnegative variables could be handled directly, and by means
of a standard variable split into the respective positive and negative
parts, general linear objective functions (with all coefficients
nonzero) could be replaced by the $\ell_1$-norm w.r.t. appropriately
rescaled variables. Since a thorough investigation of such
considerations goes beyond the scope of the present paper, we leave it
open for future research.

\appendix

\section{Active-Set Method for Linear Programs}
\label{sec:lp_active_set_method}

\subsection{Optimality Condtions for Linear Programs}
\label{sec:lp_optimality_conditions}

Let $c\in\R^{n}$, $A\in\R^{m\times n}$, $b\in\R^{m}$, $D\in\R^{k\times n}$, $e\in\R^{k}$ and $\sigma\in\set{\pm 1}^{n}$.\footnote{At this point, we use the standard notation for linear programs. The labels $A$ and $b$ appear as well in the preceding sections. However, they do not have the same meaning here.} We consider the linear program
\begin{equation}
  \label{eq:linear_program}
  \begin{aligned}
    \min_{x\in\R^{n}}&\quad & c^{\top}&x\\
    \mathrm{s.t.}&\quad & A&x \, = \, b\\
    &\quad & D&x \, \geq \, e\\
    &\quad & \diag(\sigma)&x \, \geq \, 0
  \end{aligned}
\end{equation}
and assume that it is feasible and bounded. By the well-known KKT conditions (see, e.g., \cite[Theorem~12.1]{NocedalWright2006}), $x^{*}$ is an optimal solution of (\ref{eq:linear_program}) if and only if there exist Lagrange multipliers $\lambda\in\R^{m}$, $\mu\in\R^{k}$ and $\nu\in\R^{n}$ such that the following conditions hold:
\begin{subequations}
\begin{align}
  Ax^{*} &= b\label{eq:lp_primal_feas_1}\\
  Dx^{*} &\geq e\label{eq:lp_primal_feas_2}\\
  \diag(\sigma) x^{*} &\geq 0\label{eq:lp_primal_feas_3}\\
  A^{\top}\lambda +D^{\top}\mu + \diag(\sigma)\nu&= c\label{eq:lp_dual_feas_1}\\
  \mu \odot (Dx^{*} - e) &= 0\label{eq:lp_complementarity_1}\\
  \nu \odot x^{*} &= 0\label{eq:lp_complementarity_2}\\
  \mu &\geq 0\label{eq:lp_dual_feas_2}\\
  \nu &\geq 0.\label{eq:lp_dual_feas_3}
\end{align}
\label{eq:lp_optimality_conditions}%
\end{subequations}

\subsection{General Theme}
\label{sec:generaltheme}

Suppose that $x^{\ell}\in\R^{n}$ is feasible for (\ref{eq:linear_program}), i.e., it satisfies (\ref{eq:lp_primal_feas_1})-(\ref{eq:lp_primal_feas_3}). Then, there exist non-empty sets ${\cal{A}}\subseteq\set{1,\dots,k}$ and ${\cal S}\subseteq\set{1,\dots, n}$ such that
\begin{equation*}
  D^{{\cal{A}}}x^{\ell} = e_{{\cal{A}}},\quad D^{{\cal{A}}^{c}}x^{\ell} > e_{{\cal{A}}^{c}},\quad x_{{\cal S}^{c}}^{\ell} = 0\quad \text{and}\quad \abs{x^{\ell}_{{\cal S}}} > 0.
\end{equation*}
We refer to ${\cal A}$ as the \emph{active set} and further to ${\cal S}$ as the \emph{support} of $x^{\ell}$. In the context of (\ref{eq:lp_complementarity_1}) and (\ref{eq:lp_complementarity_2}), necessarily $\mu_{{\cal A}^{c}} = 0$ and $\nu_{{\cal S}} = 0$ in case $x^{\ell}$ is an optimal solution to (\ref{eq:linear_program}). The following Lemma exploits this fact and provides alternative optimality conditions for (\ref{eq:linear_program}).

\begin{lemma}
  \label{lemma:reducedlagrangeopt}
  A point $x^{\ell}$ is an optimal solution to (\ref{eq:linear_program}) if and only if it is feasible and there exist $\lambda\in\R^{m}$ and $\mu_{{\cal A}}\in\R^{\abs{{\cal A}}}$ such that
  \begin{subequations}
    \begin{align}
      A_{{\cal S}}^{\top}\lambda + (D_{{\cal S}}^{{\cal A}})^{\top}\mu_{{\cal A}} &= c_{{\cal S}},\label{eq:reducedlagrangeopt1}\\
      \diag(\sigma_{{\cal S}^{c}})(c_{{\cal S}^{c}} - A_{{\cal S}^{c}}^{\top}\lambda - (D_{{\cal S}^{c}}^{{\cal A}})^{\top}\mu_{{\cal A}}) &\geq 0\quad\text{and}\label{eq:reducedlagrangeopt2}\\
      \mu_{{\cal A}} &\geq 0\label{eq:reducedlagrangeopt3}.
    \end{align}
  \end{subequations}
\end{lemma}
\begin{proof}
  It can easily be shown that the conditions in Lemma \ref{lemma:reducedlagrangeopt} are equivalent to conditions (\ref{eq:lp_primal_feas_1})--(\ref{eq:lp_dual_feas_3}) with $\mu_{{\cal A}^{c}} = 0$, $\nu_{{\cal S}} = 0$ and
  \begin{equation}
    \label{eq:numultiplier}
    \nu_{{\cal S}^{c}} = \diag(\sigma_{{\cal S}^{c}})(c_{{\cal S}^{c}} - A_{{\cal S}^{c}}^{\top}\lambda - (D_{{\cal S}^{c}}^{{\cal A}})^{\top}\mu_{{\cal A}}).
  \end{equation}
  \qed
\end{proof}

Starting from $x^{\ell}$, our goal is to approach a solution of (\ref{eq:linear_program}) by generating \emph{descent directions} $\xi$ that preserve the active set as well as the support and, should this not be possible, by changing these sets appropriately. We repeat these steps until we finally identify ${\cal A}$, ${\cal S}$, $\lambda$ and $\mu_{{\cal A}}$ satisfying (\ref{eq:reducedlagrangeopt1})--(\ref{eq:reducedlagrangeopt3}).

\subsection{Descent Directions and Blocking Constraints}
\label{sec:descent_directions_and_blocking_constraints}

If there exists a solution of the linear system
\begin{equation}
  \label{eq:descent_direction}
  \renewcommand{\arraystretch}{1.2}
  \begin{bmatrix}
    A_{{\cal S}}\\
    D_{{\cal S}}^{{\cal A}}\\
    c_{{\cal S}}^{\top}
  \end{bmatrix}
  \renewcommand{\arraystretch}{1}
  \xi_{{\cal S}} =
  \begin{pmatrix}
    0\\
    0\\
    -1
  \end{pmatrix}
  \quad\text{and}\quad \xi_{{\cal S}^{c}} = 0,
\end{equation}
then it holds for arbitrary $\alpha > 0$ that
\begin{equation}
  \label{eq:descent_direction_feasibility_1}
  A(x^{\ell} + \alpha\xi) = b,\quad D^{{\cal A}}(x^{\ell} + \alpha\xi) = e_{{\cal A}}\quad\text{and}\quad x_{{\cal S}^{c}}^{\ell} + \alpha \xi_{{\cal S}^{c}} = 0.
\end{equation}
The largest $\alpha > 0$ such that also
\begin{equation}
  \label{eq:descent_direction_feasibility_2}
  D^{{\cal A}^{c}}(x^{\ell} + \alpha\xi) \geq e_{{\cal A}^{c}}\quad\text{and}\quad \diag(\sigma_{{\cal S}})(x_{{\cal S}}^{\ell} + \alpha\xi_{{\cal S}}) \geq 0
\end{equation}
is given by
\begin{equation}
  \label{eq:lp_primal_step_size}
  \alpha = \min \left(\min_{i\in{\cal A}^{c}\atop d_{i}^{\top}\xi < 0} \frac{e_{i} - d_{i}^{\top}x^{\ell}}{d_{i}^{\top}\xi}, \min_{j\in{\cal S}\atop\sigma_{j}\xi_{j} < 0} -\frac{x_{j}}{\xi_{j}}\right).
\end{equation}
Note that $0 < \alpha < \infty$ since we assumed that (\ref{eq:linear_program}) is bounded. The sets
\begin{equation}
  \label{eq:blocking_constraints}
  {\cal A}^{+} = \set{i \in {\cal A}^{c} : d_{i}^{\top}(x^{\ell} + \alpha\xi) = e_{i}} \quad \text{and} \quad {\cal S}^{-} = \set{j \in {\cal S} : x^{\ell}_{j} + \alpha\xi_{j} = 0}
\end{equation}
are the index sets where the minimum is attained, i.e., the sets of \emph{blocking constraints}. Each $i\in{\cal A}^{+}$ joins the active set and each $j\in{\cal S}^{-}$ leaves the support if we perform the step $\alpha\xi$. Consequently, we update $x^{\ell+1} = x^{\ell} + \alpha\xi$, ${\cal A} = {\cal A}\cup{\cal A}^{+}$ and ${\cal S} = {\cal S}\setminus{\cal S}^{-}$.

\subsection{Lagrange Multipliers}
\label{sec:lagrange_multipliers}

If there is no direction according to (\ref{eq:descent_direction}), then zero is an optimal solution of
\begin{equation}
  \label{eq:subproblem}
  \begin{aligned}
    \min_{\xi_{{\cal S}}\in\R^{\abs{{\cal S}}}} && c_{{\cal S}}^{\top}&\xi_{{\cal S}}\\
    \mathrm{s.t.} &&
    \begin{bmatrix}
      A_{{\cal S}}\\
      D_{{\cal S}}^{{\cal A}}
    \end{bmatrix}
    &\xi_{{\cal S}} = 0
  \end{aligned}
\end{equation}
Employing KKT conditions again, we see that there exist $\lambda$ and $\mu_{{\cal A}}$ satisfying \eqref{eq:reducedlagrangeopt1}. For the case that $\lambda$ and $\mu_{{\cal A}}$ additionally satisfy \eqref{eq:reducedlagrangeopt2}--\eqref{eq:reducedlagrangeopt3}, Lemma \ref{lemma:reducedlagrangeopt} states that $x^{\ell}$ is an optimal solution.

Otherwise, with $\nu_{{\cal S}^{c}}$ according to \eqref{eq:numultiplier}, there exists at least one index $i\in{\cal A}$ such that $\mu_{i} < 0$ or $j\in{\cal S}^{c}$ such that $\nu_{j} < 0$. We select the smaller of both values and set ${\cal A} = {\cal A} \setminus \set{i}$ or ${\cal S} = {\cal S}\cup\set{j}$, respectively. Then, we search a new direction according to Subsection \ref{sec:descent_directions_and_blocking_constraints}.

\subsection{Feasibility of Generated Directions}
\label{sec:feasibilityofgenerateddirections}

In the context of the previous section, suppose that $\mu_{i} < 0$ and we set ${\cal A} = {\cal A} \setminus\set{i}$. Afterwards, we go back to \eqref{eq:descent_direction} and find a direction $\xi$. It holds that
\begin{equation}
  \label{eq:direction_feasibility}
  \begin{aligned}
    -1 \overset{(\ref{eq:descent_direction})}{=} c_{{\cal S}}^{\top}\xi_{{\cal S}} &\overset{(\ref{eq:reducedlagrangeopt1})}{=} (A_{{\cal S}}^{\top}\lambda + (D_{{\cal S}}^{{\cal A}})^{\top}\mu_{{\cal A}} + (D_{{\cal S}}^{i})^{\top}\mu_{i})^{\top}\xi_{{\cal S}}\\
    &\hspace{1.825mm}= \lambda^{\top}A_{{\cal S}}\xi_{{\cal S}} + \mu_{{\cal A}}^{\top}D_{{\cal S}}^{{\cal A}}\xi_{{\cal S}} + \mu_{i}D_{{\cal S}}^{i}\xi_{{\cal S}}\\
    &\hspace{.8mm}\overset{(\ref{eq:descent_direction})}{=} \mu_{i}d_{i}^{\top}\xi.
  \end{aligned}
\end{equation}
It follows that $d_{i}^{\top}\xi = -\mu_{i}^{-1} > 0$. Consequently, it holds that $d_{i}^{\top}(x^{\ell} + \alpha\xi) > e_{i}$ and the step $\alpha\xi$ preserves the property of ${\cal A}$ exactly reflecting the set of active constraints. An analogous statement holds if we update ${\cal S} = {\cal S}\cup\set{j}$ prior to finding a direction $\xi$. In that case, we obtain $\sigma_{j}\xi_{j} = -\nu_{j}^{-1} > 0$.
\medskip

Note that, if we found $\mu_{\set{i,i'}} < 0$ for distinct indices $i, i'\in{\cal A}$ and set ${\cal A} = {\cal A}\setminus\set{i,i'}$, we would not necessarily get $d_{i}^{\top}\xi > 0$ and $d_{i'}^{\top}\xi > 0$. Repeating the above reasoning only shows $(\mu_{i}d_{i} + \mu_{i'}d_{i'})^{\top}\xi > 0$. The same holds if we have $\nu_{\set{j, j'}} < 0$ or $\mu_{i} < 0$ and $\nu_{j} < 0$. Therefore, we do not change more than one index before we search for a new direction. However, it can occur that we do not immediately find a new direction after changing one index in ${\cal A}$ or ${\cal S}$. In that case, we have two determine Lagrange multipliers repeatedly and change ${\cal A}$ and ${\cal S}$ until we are able to find a new direction. This situation needs to be handled with care in order to correctly keep track of ${\cal A}$ and ${\cal S}$. We capture this aspect in Appendix \ref{sec:algorithm}.
\medskip

\subsection{Fixing New Support Variables}
\label{sec:fixingvariables}

Equation (\ref{eq:direction_feasibility}) further shows that, if we replace $c_{{\cal S}}^{\top}\xi_{{\cal S}} = -1$ by $d_{i}^{\top}\xi = 1$ in (\ref{eq:descent_direction}), this implies $c^{\top}\xi = \mu_{i} < 0$. The resulting system is
\begin{equation}
  \label{eq:direction_equation_replaced}
  \begin{bmatrix}
    A^{{\cal S}}\\
    D^{{\cal S}}_{{\cal A}}\vspace{.5mm}\\
    D^{{\cal S}}_{i}
  \end{bmatrix}
  \xi_{{\cal S}} =
  \begin{pmatrix}
    0\\
    0\\
    1
  \end{pmatrix}.
\end{equation}
Numerically, there is no obvious gain in the replacement of one equation. Essentially, the new constraint specifies $d_{i}^{\top}\xi = 1$. The same reasoning for the case that $j\in{\cal S}$ was recently added to the support shows that by dropping $c_{{\cal S}}^{\top}\xi_{{\cal S}} = -1$ and fixing $\xi_{j} = \sigma_{j}$, we obtain $c^{\top}\xi = \nu_{j} < 0$. Considering the numerical effort, this can be beneficial since we not only drop a constraint but also reduce the number of variables in the system. The result is
\begin{equation}
  \label{eq:direction_fixed_variable}
  \begin{bmatrix}
    A^{{\cal S}\setminus\set{j}}\\
    D^{{\cal S}\setminus\set{j}}_{{\cal A}}
  \end{bmatrix}
  \xi_{{\cal S}\setminus\set{j}} = -\sigma_{j}
  \begin{pmatrix}
    A^{j}\\
    D_{{\cal A}}^{j}
  \end{pmatrix}.
\end{equation}

\subsection{Algorithm and Implementation of $\ell_1$-HOUDINI}
\label{sec:algorithm}

Algorithm \ref{alg:lpactiveset} illustrates the iterative scheme discussed in Appendix \ref{sec:generaltheme}--\ref{sec:fixingvariables}. Additionally, we assume that an initial direction $\xi$ is provided as input since this is the situation we are faced with in Section \ref{sec:practicalconsiderations}.

The conditional statement beginning in Step \ref{step:special_cases} considers two special cases. In that context, ${\cal A}^{-}$ is the set of indices that were consecutively removed from the active set in Steps \ref{step:lagrange_multipliers_begin}--\ref{step:lagrange_multipliers_end} and ${\cal S}^{+}$ is the set of indices that were consecutively added to the support. It can occur that $\abs{{\cal A}^{-}} + \abs{{\cal S}^{+}} > 1$ in case we do not find a direction in Step \ref{step:direction} in a positive number of consecutive iterations.

The first case is $\alpha = 0$ which can occur if $\abs{{\cal A}^{-}} + \abs{{\cal S}^{+}} > 1$ and there exists $i\in{\cal A}^{-}$ such that $d_{i}^{\top}\xi < 0$ or $\sigma_{j}\xi_{j} < 0$ for some $j\in{\cal S}^{+}$. The respective indices are re-added to ${\cal A}$ and re-removed from ${\cal S}$, respectively, before trying to find a new feasible direction.

In the second case, if $\alpha > 0$ and $\abs{{\cal A}^{-}} + \abs{{\cal S}^{+}} > 1$, we can still have $i\in{\cal A}^{-}$ with $d_{i}^{\top}\xi = 0$ or $\sigma_{j}\xi_{j} = 0$ for some $j\in{\cal S}^{+}$. Consequently, the $i$-th constraint stays active and $j$ does not join the support after a step in direction $\xi$. We adapt ${\cal A}$ and ${\cal S}$ accordingly. Since we have performed a non-zero step, we moreover reset ${\cal A}^{-}$ and ${\cal S}^{+}$.
\medskip

Table \ref{tab:primalupdate} puts the primal update from Section \ref{sec:activesetprimal} into the context of Algorithm \ref{alg:lpactiveset}. Notice that problem (\ref{eq:primalactivesetproblem}) needs to be reformulated as a minimization problem in order to have the form (\ref{eq:linear_program}). Table \ref{tab:dualupdate} does the same for the dual update from Section \ref{sec:activesetdual}.

In both the primal and the dual case we applied some easy sign substitutions in order to bring (\ref{eq:reducedlagrangeopt1}) into a simple form. Of course, the respective inverse substitutions appear in the formulas for $\mu_{{\cal S}}$ and $\nu_{{\cal S}^{c}}$, respectively.

Moreover, we used that during the primal update $\sign(y^{k+1}_{I_{D}}) = \sign(A^{I_{D}}\xi^{\ell} - b_{I_{D}})$ throughout.

\SetAlFnt{\normalsize}
\DontPrintSemicolon
\begin{algorithm}
  \caption{Active-Set Method for LPs.}
  \label{alg:lpactiveset}

  \KwIn{$c\in\R^{n}$, $A\in\R^{m\times n}$, $b\in\R^{m}$, $D\in\R^{k\times n}$, $e\in\R^{k}$, $\sigma\in\set{\pm 1}^{n}$, feasible $x^{0}\in\R^{n}$ and associated sets ${\cal A}$ and ${\cal S}$, initial direction $\xi$}
  \KwOut{solution $x^{*}$ to problem (\ref{eq:linear_program})}   
  \BlankLine
  \BlankLine
  $\ell\leftarrow 0$\;
  \While{not stopped}{
    \If{a solution $\xi$ of \eqref{eq:descent_direction} exists}{\label{step:direction}
      $\alpha\leftarrow$ step size according to (\ref{eq:lp_primal_step_size})\;
      $x^{\ell + 1} \leftarrow x^{\ell} + \alpha \xi$\;
      $({\cal A}^{+}, {\cal S}^{-})\leftarrow$ blocking constraints according to (\ref{eq:blocking_constraints})\;
      ${\cal A}\leftarrow{\cal A}\cup{\cal A}^{+}$\;
      ${\cal S}\leftarrow{\cal S}\setminus{\cal S}^{-}$\;
      \If{$\alpha = 0$}{\label{step:special_cases}
        ${\cal A}^{-}\leftarrow {\cal A}^{-}\setminus {\cal A}^{+}$\;
        ${\cal S}^{+}\leftarrow {\cal S}^{+}\setminus {\cal S}^{-}$\;
      }
      \ElseIf{$\abs{{\cal A}^{-}} + \abs{{\cal S}^{+}} > 1$}{
        ${\cal A}\leftarrow{\cal A}\cup\set{i\in{\cal A}^{-} : d_{i}^{\top}\xi = 0}$\;
        ${\cal S}\leftarrow{\cal S}\setminus\set{j\in{\cal S}^{+} : \xi_{j} = 0}$\;
        ${\cal A}^{-}\leftarrow\emptyset$\;
        ${\cal S}^{+}\leftarrow\emptyset$\;
      }
    }
    $\ell\leftarrow \ell + 1$\;
    \Else{
      $(\mu_{{\cal A}}, \nu_{{\cal S}^{c}})\leftarrow$ Lagrange multipliers according to (\ref{eq:reducedlagrangeopt1}) and (\ref{eq:numultiplier})\;
      $i^{-}\leftarrow\argmin_{i\in{\cal A}} \mu_{i}$\;
      $j^{+}\leftarrow\argmin_{j\in{\cal S}^{c}} \nu_{j}$\;
      \If{$\mu_{i^{-}} \geq 0$ and $\nu_{j^{+}} \geq 0$}{\label{step:lagrange_multipliers_begin}
        \Return{$x^{*} = x^{\ell}$}\;
      }
      \ElseIf{$\mu_{i^{-}} < \nu_{j^{+}}$}{
        ${\cal A}\leftarrow {\cal A}\setminus\set{i^{-}}$\;
        ${\cal A}^{-}\leftarrow {\cal A}^{-} \cup \set{i^{-}}$\;
      }
      \Else{
        ${\cal S}\leftarrow {\cal S}\cup\set{j^{+}}$\;
        ${\cal S}^{+}\leftarrow {\cal S}^{+} \cup \set{j^{+}}$\;\label{step:lagrange_multipliers_end}
      }
    }
  }
\end{algorithm}

\renewcommand{\arraystretch}{1.8}
\begin{table}
  \begin{tabular}{|L{0.126\textwidth} C{0.8\textwidth}|}
    \hline
    ${\cal S}$ &$I_{D}$\\
    ${\cal A}$ &$J_{D}\setminus J_{P}$\\
    $A_{{\cal S}}$ &$(-A_{J_{P}}^{I_{D}})^{\top}$\\
    $D_{{\cal S}}^{{\cal A}}$ &$(A^{\top}_{J_{D}\setminus J_{P}}\psi^{\ell}) \odot (-A^{I_{D}}_{J_{D}\setminus J_{P}})^{\top}$\\
    $c_{{\cal S}}$ &$-\sign(A^{I_{D}}x^{k} - b_{I_{D}})$\\
    (\ref{eq:descent_direction}) &$\begin{aligned}(A^{I_{D}}_{J_{D}})^{\top}e_{I_{D}} &= 0\\\sign(A^{I_{D}}x^{k} - b_{I_{D}})^{\top}e_{I_{D}} &= 1\end{aligned}$\\
    (\ref{eq:reducedlagrangeopt1}) &$A^{I_{D}}_{J_{D}} \hat d_{J_{D}} = -\sign(A^{I_{D}}x^{k} - b_{I_{D}})$\\
    $\mu_{{\cal A}}$ &$-(A^{\top}_{J_{D}\setminus J_{P}}\psi^{\ell}) \odot \hat d_{J_{D}\setminus J_{P}}$\\
    $\sigma_{{\cal S}^{c}}$ &$\sign(A^{I_{P}\setminus I_{D}}x^{k} - b_{I_{P}\setminus I_{D}})$\\
    $c_{{\cal S}^{c}}$ &$-\sign(A^{I_{P}\setminus I_{D}}x^{k} - b_{I_{P}\setminus I_{D}})$\\
    $A_{{\cal S}^{c}}$ &$(-A^{I_{P}\setminus I_{D}}_{J_{P}})^{\top}$\\
    $D_{{\cal S}^{c}}^{{\cal A}}$ &$(A_{J_{D}\setminus J_{P}}^{\top}\psi^{\ell})\odot(-A^{I_{P}\setminus I_{D}}_{J_{D}\setminus J_{P}})^{\top}$\\
    $\nu_{{\cal S}^{c}}$ &$- \sign(A^{I_{P}\setminus I_{D}}x^{k} - b_{I_{P}\setminus I_{D}}) \odot A^{I_{P}\setminus I_{D}}\hat d - \One$\\\hline
  \end{tabular}
  \caption{Active-Set Implementation of the Dual Update.}
  \label{tab:dualupdate}
\end{table}
\renewcommand{\arraystretch}{1}

\renewcommand{\arraystretch}{1.8}
\begin{table}
  \begin{tabular}{|L{0.126\textwidth} C{0.8\textwidth}|}
    \hline
    ${\cal S}$&$J_{P} \cup \set{t}$\\
    ${\cal A}$&$I_{P}\setminus I_{D}$\\
    $A_{{\cal S}}$ &$\begin{bmatrix}A_{J_{P}}^{I_{D}} \quad \sign(y_{I_{D}}^{k+1})\end{bmatrix}$\\
    $D_{{\cal S}}^{{\cal A}}$ &$\begin{bmatrix}-\sign(A^{I_{P}\setminus I_{D}}\xi^{\ell} - b_{I_{P}\setminus I_{D}}) \odot A_{J_{P}}^{I_{P}\setminus I_{D}} \quad -\One\end{bmatrix}$\\
    $c_{{\cal S}}$ &$(0, -1)^{\top}$\\
    (\ref{eq:descent_direction})&$A^{I_{P}}_{J_{P}}d_{J_{P}} = -\sign(A^{I_{P}}\xi^{\ell} - b_{I_{P}})$\\
    (\ref{eq:reducedlagrangeopt1})&$\begin{aligned}(A^{I_{P}}_{J_{P}})^{\top} &\hat e_{I_{P}} = 0\\ \sign(A^{I_{P}}\xi^{\ell} - b_{I_{P}})^{\top} &\hat e_{I_{P}} = 1\end{aligned}$\\
    $\mu_{{\cal A}}$ & $\sign(A^{I_{P}\setminus I_{D}}\xi^{\ell} - b_{I_{P}\setminus I_{D}})\odot \hat e_{I_{P}\setminus I_{D}}$\\
    $\sigma_{{\cal S}^{c}}$ &$A_{J_{P}\setminus J_{D}}^{\top}y^{k+1}$\\
    $c_{{\cal S}^{c}}$ &$0$\\
    $A_{{\cal S}^{c}}$ &$A_{J_{D}\setminus J_{P}}^{I_{D}}$\\
    $D_{{\cal S}^{c}}^{{\cal A}}$&$-\sign(A^{I_{P}\setminus I_{D}}\xi^{\ell} - b_{I_{P}\setminus I_{D}}) \odot A_{J_{D}\setminus J_{P}}^{I_{P}\setminus I_{D}}$\\
    $\nu_{{\cal S}^{c}}$ &$-(A_{J_{D}\setminus J_{P}}^{\top}y^{k+1})\odot A_{J_{D}\setminus J_{P}}^{\top}\hat e$\\\hline
  \end{tabular}
  \caption{Active-Set Implementation of the Primal Update.}
  \label{tab:primalupdate}
\end{table}
\renewcommand{\arraystretch}{1}

\clearpage

%
%


\bibliographystyle{spmpsci}      
\bibliography{references.bib}   

%
%

\end{document}